\documentclass[12pt]{article}
\usepackage{graphicx}
\usepackage{hyperref}
\usepackage{mathrsfs}
\usepackage{amsfonts,amsmath,amsthm, amssymb}
\usepackage{latexsym, euscript, epic, eepic, color}
\usepackage{indentfirst}
\usepackage{color}
\usepackage{enumerate}

\makeatletter
\@addtoreset{equation}{section}
\makeatother
\newtheorem{theorem}{Theorem}[section]
\newtheorem{lemma}[theorem]{Lemma}

\newtheorem*{dthm}{Dirichlet's theorem}

\newtheorem{corollary}[theorem]{Corollary}

\theoremstyle{definition}

\newtheorem{remark}[theorem]{Remark}
\makeatletter

\newcommand{\Rmnum}[1]{\expandafter\@slowromancap\romannumeral #1@}
\makeatother

\begin{document}
\title{Uniform Diophantine approximation with restricted denominators}

\author{Bo Wang\textsuperscript{ a}, Bing Li\textsuperscript{ a} and Ruofan Li\textsuperscript{ b,}\thanks{Corresponding author} \\
\small \it \textsuperscript{\rm a }School of Mathematics, South China University of Technology,\\
\small \it  Guangzhou, China, 510641\\
\small \it \textsuperscript{\rm b }Department of Mathematics, Jinan University, Guangzhou, China, 510632}
\date{\today}

\maketitle

\begin{center}
\begin{minipage}{120mm}
{\small {\bf Abstract.} Let $b\geq2$ be an integer and $A=(a_{n})_{n=1}^{\infty}$ be a strictly increasing subsequence of positive integers with $\eta:=\limsup\limits_{n\to\infty}\frac{a_{n+1}}{a_{n}}<+\infty$. For each irrational real number $\xi$, we denote by $\hat{v}_{b,A}(\xi)$ the supremum of the real numbers $\hat{v}$ for which, for every sufficiently large integer $N$, the equation $\|b^{a_n}\xi\|<(b^{a_N})^{-\hat{v}}$ has a solution $n$ with $1\leq n\leq N$. For every $\hat{v}\in[0,\eta]$, let $\hat{\mathcal{V}}_{b,A}(\hat{v})$ ($\hat{\mathcal{V}}_{b,A}^{\ast}(\hat{v})$) be the set of all real numbers $\xi$ such that $\hat{v}_{b,A}(\xi)\geq\hat{v}$ ($\hat{v}_{b,A}(\xi)=\hat{v}$) respectively. In this paper, we give some results of the Hausdorfff dimensions of $\hat{\mathcal{V}}_{b,A}(\hat{v})$ and $\hat{\mathcal{V}}_{b,A}^{\ast}(\hat{v})$. When $\eta=1$, we prove that the Hausdorfff dimensions of $\hat{\mathcal{V}}_{b,A}(\hat{v})$ and $\hat{\mathcal{V}}_{b,A}^{\ast}(\hat{v})$ are equal to $\left(\frac{1-\hat{v}}{1+\hat{v}}\right)^{2}$ for any $\hat{v}\in[0,1]$. When $\eta>1$ and $\lim_{n\to\infty}\frac{a_{n+1}}{a_{n}}$ exists, we show that the Hausdorfff dimension of $\hat{\mathcal{V}}_{b,A}(\hat{v})$ is strictly less than $\left(\frac{\eta-\hat{v}}{\eta+\hat{v}}\right)^{2}$ for some $\hat{v}$, which is different with the case $\eta=1$, and we give a lower bound of the Hausdorfff dimensions of $\hat{\mathcal{V}}_{b,A}(\hat{v})$ and $\hat{\mathcal{V}}_{b,A}^{\ast}(\hat{v})$ for any $\hat{v}\in[0,\eta]$. Furthermore, we show that this lower bound can be reached for some $\hat{v}$.}
\end{minipage}
\end{center}

\footnotetext{\small \it E-mails addresses:  \rm 202010106047@mail.scut.edu.cn (B. Wang), scbingli@scut.edu.cn (B. Li), liruofan@jnu.edu.cn (R. Li). }

\section{Introduction}
Throughout this paper, $\|\cdot\|$ stands for the distance to the nearest integer. In 1842, Dirichlet \cite{L} proved the following theorem.
\begin{dthm}
For any $\xi\in\mathbb{R}$ and every real number $X\geq1$, there exists an integer with $1\leq n\leq X$, such that
\begin{equation*}
\|n\xi\|<\frac{1}{X}.
\end{equation*}
\end{dthm}
Following Waldschmidt \cite{MW}, we call Dirichlet's theorem a uniform approximation theorem. A weak form of it, called an asymptotic approximation theorem, was already known (see Legendre's book \cite{AML}) before Dirichlet: for any $\xi\in\mathbb{R}$, there exists infinitely many integers $n$ such that $\|n\xi\|<\frac{1}{n}$. Following the notations introduced in \cite{YM}, we denote by $\hat{w}_{1}(\xi)$ the supremum of the real number $\hat{w}$ such that, for all sufficiently large real number $X$, there exists $1\leq n\leq X$, satisfying $\|n\xi\|<X^{-w}$. And we denote by $w_{1}(\xi)$ the supremum of the real number $w$ such that, there exists infinitely many integers $n$, satisfying $\|n\xi\|<n^{-w}$. It follows form Dirichlet's theorem that $w_{1}(\xi)\geq\hat{w}_{1}(\xi)\geq1$. In 1926, Khintchine \cite{A} proved that $\hat{w}_{1}(\xi)=1$ for every irrational real number $\xi$. What is more, by the first Borel-Cantelli lemma \cite{GH}, we also know that $w_{1}(\xi)=1$ for almost all real numbers $\xi$. Here and below, ``almost all" always refers to the Lebesgue measure. Furthermore, a classical result, established independently by Jarn\'{\i}k \cite{VJ} and Besicovitch \cite{ASB}, asserts that
\begin{equation*}
\dim_{\rm H}\left(\{\xi\in\mathbb{R}: w_{1}(\xi)\geq w\}\right)=\frac{2}{w+1}\ {\rm for\ all}\ w\geq1,
\end{equation*}
where $\dim_{\rm H}$ stands for the Hausdorff dimension (we refer the reader to $\S\ref{20}$ for terminologies).
By Amou and Bugeaud's result \cite[Theorem 7]{MY}, the Hausdorff dimension of the set $\{\xi\in\mathbb{R}: w_{1}(\xi)=w\}$ also equals to $\frac{2}{w+1}$. Moreover, let $\psi:[X_{0},+\infty)\to\mathbb{R}_{>0}$ be a decreasing function with $X_{0}\geq1$ fixed. Here and throughout, $\mathbb{R}_{>0}=(0,+\infty)$. We denote by $D(\psi)$ the set of all real numbers $\xi$ such that for all sufficiently large real number $X$, there exists $1\leq n\leq X$, satisfying $\|n\xi\|<\psi(X)$. We call $D(\psi)$ the Dirichlet improvable set. Denote by $D(\psi)^{c}$ the complement of $D(\psi)$, we call $D(\psi)^{c}$ the Dirichlet non-improvable set. Let $\psi_{1}(X)=X^{-1}$. By Dirichlet's theorem, we know that $D(\psi_{1})=\mathbb{R}$. A classical result of Davenport and Schmidt \cite{HW} shows that for any $0<\varepsilon<1$, $D(\varepsilon\psi_{1})$ is a subset of the union of $\mathbb{Q}$ and the set of badly approximable numbers. Thus, $D(\varepsilon\psi_{1})^{c}$ has full Lebesgue measure. Furthermore, by Kleinbock and Wadleigh's result \cite[Theorem 1.7]{DN}, $D(\psi)^{c}\neq\emptyset$ whenever $\psi$ is decreasing and
\begin{equation*}
X\psi(X)<1\ {\rm for\ all\ large}\ X.
\end{equation*}
Moreover, they also gave the result of the Lebesgue measure of $D(\psi)^{c}$. To state their result, let
\begin{equation}\label{10}
\Psi(X):=\frac{X\psi(X)}{1-X\psi(X)}
\end{equation}
and denote by $\lambda$ the Lebesgue measure. Let $\psi$ be a decreasing non-negative function such that $\Psi$ is increasing and $X\psi(X)<1$ for all $X\geq X_{0}$, by \cite[Theorem 1.8]{DN}, we know that if
\begin{equation*}
\sum_{X}\frac{\log\Psi(X)}{X\Psi(X)}<\infty\ ({\rm respectively}=\infty),
\end{equation*}
then
\begin{equation*}
\lambda(D(\psi)^{c})=0\ ({\rm respectively}\ \lambda(D(\psi))=0).
\end{equation*}
Let $f$ be a dimension function, Hussain, Kleinbock, Wadleigh and Wang \cite[Theorem 1.6]{MDNB} have established the following result, which gives the $f$-dimensional Hausdorff measure of $D(\psi)^{c}$ for a large class of dimension functions. Let $\psi(X)$ be a decreasing non-negative function with $X\psi(X)<1$ for all large $X$. If a dimension function $f$ is essentially sub-linear, that is,
\begin{equation*}
{\rm there\ exists}\ B>1\ {\rm such\ that}\ \limsup_{x\to0}\frac{f(Bx)}{f(x)}<B,
\end{equation*}
then
\begin{equation*}
\mathcal{H}^{f}(D(\psi)^{c})=\begin{cases} 0, &{\rm if}\ \sum\limits_{X}Xf(\frac{1}{X^{2}\Psi(X)})<\infty, \\ \infty, &{\rm if}\ \sum\limits_{X}Xf(\frac{1}{X^{2}\Psi(X)})=\infty. \end{cases}
\end{equation*}
Hence,
\begin{equation*}
\dim_{\rm H}(D(\psi)^{c})=\frac{2}{2+\tau},\ {\rm where}\ \tau=\liminf\limits_{X\to\infty}\frac{\log\Psi(X)}{\log X}.
\end{equation*}
Furthermore, if $f$ is a non-essentially sub-linear dimension function, Bos, Hussain and Simmons \cite[\rm Theorem 1.6]{PMD} showed that
\begin{equation*}
\mathcal{H}^{f}(D(\psi)^{c})=\begin{cases} 0, &{\rm if}\ \sum\limits_{X}X\log(\Psi(X))f(\frac{1}{X^{2}\Psi(X)})<\infty, \\ \infty, &{\rm if}\ \sum\limits_{X}X\log(\Psi(X))f(\frac{1}{X^{2}\Psi(X)})=\infty. \end{cases}
\end{equation*}

In fact, there is some connection between the Dirichlet non-improvable set $D(\psi)^{c}$ and the $\psi$ well-approximable set $W(\psi)$, which is the set of all real numbers $\xi$ such that $\|n\xi\|<\psi(n)$ for infinitely many positive integers $n$. Let $\xi=[a_{0}(\xi);a_{1}(\xi),a_{2}(\xi),\ldots]$ be the continued fraction of $\xi$, and $\frac{p_{n}(\xi)}{q_{n}(\xi)}$ be the $n$-th convergent of $\xi$. Let $\Phi_{1},\Phi_{2}:[1,+\infty)\to\mathbb{R}_{>0}$ be two non-decreasing positive functions, denote
\begin{equation*}
\mathcal{K}(\Phi_{2}):=\{\xi\in\mathbb{R}:a_{n+1}(\xi)\geq\Phi_{2}(q_{n}(\xi))\ {\rm for\ infinitely\ many}\ n\in\mathbb{N}\}
\end{equation*}
and
\begin{equation*}
\mathcal{G}(\Phi_{1}):=\{\xi\in\mathbb{R}:a_{n}(\xi)a_{n+1}(\xi)\geq\Phi_{1}(q_{n}(\xi))\ {\rm for\ infinitely\ many}\ n\in\mathbb{N}\}.
\end{equation*}
When
\begin{equation*}
\Phi_{2}(X)=\frac{1}{\psi(X)X}\ {\rm and}\ \Phi_{1}(X)=\frac{\psi(X)X}{1-\psi(X)X},
\end{equation*}
by some elementary calculations (see \cite[\rm Lemma 2.2]{DN} or \cite[\rm p.2-3]{BBJ}), one has the inclusions
\begin{equation*}
\mathcal{K}(\Phi_{2})\subset W(\psi)\subset\mathcal{K}\left(\frac{1}{2}\Phi_{2}\right)\ {\rm and}\ \mathcal{G}(\Phi_{1})\subset D(\psi)^{c}\subset\mathcal{G}\left(\frac{1}{4}\Phi_{1}\right).
\end{equation*}
When $\Phi_{2}=\kappa\Phi_{1}$, where $\kappa$ is a positive constant, Bakhtawar, Bos and Hussain \cite[\rm Theorem 1.5]{APM} proved that the Hausdorff dimension of $\mathcal{G}(\Phi_{1})\setminus\mathcal{K}(\Phi_{2})$ is equal to the Hausdorff dimension of $\mathcal{K}(\Phi_{2})$.
Furthermore, Li, Wang and Xu \cite[\rm Proposition 1.7 and Theorem 1.8]{BBJ} also studied the Hausdorff dimension of $\mathcal{G}(\Phi_{1})\setminus\mathcal{K}(\Phi_{2})$ when $\lim\limits_{X\to\infty}\frac{\log\Phi_{1}(X)}{\log X}$ and $\lim\limits_{X\to\infty}\frac{\log\Phi_{2}(X)}{\log X}$ both exists. It is worth mentioning that they showed that the Hausdorff dimension of $\mathcal{G}(\Phi_{1})\setminus\mathcal{K}(\Phi_{2})$ will change greatly even if slightly modifying $\Phi_{1}$ by a constant. We refer the reader to \cite{BBJ} for more details.

To our knowledge, there are no complete results about the Hausdorff measure or the Hausdorff dimension of $D(\psi)$. However, Bugeaud and Liao \cite{YL} gave a result of the uniform Diophantine approximation related to $b$-ary expansions, where $b\geq2$ is an integer. To state their result, let us give some notations. For each irrational real number $\xi$, following the notation introduced in \cite{MY}, we denote by $v_{b}(\xi)$ the supremum of the real numbers $v$ for which the equation
\begin{equation*}
\|b^{n}\xi\|<(b^{n})^{-v}
\end{equation*}
has infinitely many solutions in positive integers. Following the notation introduced in \cite{YL}, we denote by $\hat{v}_{b}(\xi)$ the supremum of the real numbers $\hat{v}$ for which, for every sufficiently large integer $N$, the equation
\begin{equation*}
\|b^{n}\xi\|<(b^{N})^{-\hat{v}}
\end{equation*}
has a solution $n$ with $1\leq n\leq N$.

By the theory of continued fractions, we know that $\hat{v}_{b}(\xi)\leq1$. An easy covering argument implies that the set
\begin{equation*}
\{\xi\in\mathbb{R}:v_{b}(\xi)=0\}
\end{equation*}
has full Lebesgue measure. Since
\begin{equation*}
0\leq\hat{v}_{b}(\xi)\leq v_{b}(\xi),
\end{equation*}
we have $\hat{v}_{b}(\xi)=0$ for almost all real numbers $\xi$. For any $v\in \mathbb{R}_{\geq0}\cup\{+\infty\}$, let
\begin{equation*}
\mathcal{V}_{b}(v)=\{\xi\in\mathbb{R}:v_{b}(\xi)\geq v\}\ {\rm and}\ \mathcal{V}_{b}^{\ast}(v)=\{\xi\in\mathbb{R}:v_{b}(\xi)=v\}.
\end{equation*}
For any $\hat{v}\in[0,1]$, let
\begin{equation*}
\hat{\mathcal{V}}_{b}(\hat{v})=\{\xi\in\mathbb{R}:\hat{v}_{b}(\xi)\geq\hat{v}\}\ {\rm and}\ \hat{\mathcal{V}}_{b}^{\ast}(\hat{v})=\{\xi\in\mathbb{R}:\hat{v}_{b}(\xi)=\hat{v}\}.
\end{equation*}
By a general result of Borosh and Fraenkel \cite{IA} or the mass transference principle of Beresnevich and Velani \cite{VS}, for any $v\in\mathbb{R}_{\geq0}\cup\{+\infty\}$, we have
\begin{equation*}
\dim_{\rm H}(\mathcal{V}_{b}(v))=\frac{1}{1+v}.
\end{equation*}
What is more, it follows from \cite[Theorem 7]{MY} that the Hausdorff dimension of $\mathcal{V}_{b}^{\ast}(v)$ is also equal to $\frac{1}{1+v}$ for any $v\in\mathbb{R}_{\geq0}\cup\{+\infty\}$. In 2016, Bugeaud and Liao \cite[Theorem 1.1]{YL} have established that
\begin{equation*}
\dim_{\rm H}(\hat{\mathcal{V}}_{b}^{\ast}(\hat{v}))=\dim_{\rm H}(\hat{\mathcal{V}}_{b}(\hat{v}))=\left(\frac{1-\hat{v}}{1+\hat{v}}\right)^{2}
\end{equation*}
for all $\hat{v}\in[0,1]$.

In this paper, we restrict our attention to approximation by rational numbers whose denominator is a restricted power of some given integer $b\geq2$ and consider the following exponents of approximation. Let $A=(a_{n})_{n=1}^{\infty}\subset\mathbb{N}$ be a strictly increasing sequence with
\begin{equation*}
\eta=\eta(A):=\limsup_{n\to\infty}\frac{a_{n+1}}{a_{n}}<+\infty.
\end{equation*}
For a irrational real number $\xi$, similar to \cite{MY} and \cite{YL}, we define
\begin{equation*}
v_{b,A}(\xi):=\sup\{v\geq0: \|b^{a_{n}}\xi\|<(b^{a_{n}})^{-v}\ {\rm for\ i.m.}\ n\in\mathbb{N}\}
\end{equation*}
and
\begin{equation*}
\hat{v}_{b,A}(\xi):=\sup\left\{\hat{v}\geq0: \forall N\gg1, \min_{1\leq n\leq N}\|b^{a_{n}}\xi\|<(b^{a_{N}})^{-\hat{v}}\right\}.
\end{equation*}
Here and throughout, ``i.m.'' stands for ``infinitely many'' and ``$N\gg1$'' means ``$N$ large enough''. In our new notations, we have
\begin{equation*}
v_{b}(\xi)=v_{b,\mathbb{N}}(\xi)
\end{equation*}
and
\begin{equation*}
\hat{v}_{b}(\xi)=\hat{v}_{b,\mathbb{N}}(\xi).
\end{equation*}
By the definition of $v_{b,A}(\xi)$ and $\hat{v}_{b,A}(\xi)$, we know that
\begin{equation*}
0\leq\hat{v}_{b,A}(\xi)\leq v_{b,A}(\xi).
\end{equation*}
Furthermore, we have
\begin{equation}\label{17}
v_{b,A}(\xi)\leq v_{b}(\xi)
\end{equation}
and
\begin{equation}\label{18}
\hat{v}_{b,A}(\xi)\leq\eta\cdot\hat{v}_{b}(\xi)\leq\eta.
\end{equation}
Given $v\in\mathbb{R}_{\geq0}\cup\{+\infty\}$, let
\begin{equation*}
\mathcal{V}_{b,A}(v):=\{\xi\in\mathbb{R}:v_{b,A}(\xi)\geq v\}\ {\rm and}\ \mathcal{V}_{b,A}^{\ast}(v):=\{\xi\in\mathbb{R}: v_{b,A}(\xi)=v\}.
\end{equation*}
Given $\hat{v}\in[0,\eta]$, let
\begin{equation*}
\hat{\mathcal{V}}_{b,A}(\hat{v}):=\{\xi\in\mathbb{R}: \hat{v}_{b,A}(\xi)\geq\hat{v}\}\ {\rm and}\ \hat{\mathcal{V}}_{b,A}^{\ast}(\hat{v}):=\{\xi\in\mathbb{R}: \hat{v}_{b,A}(\xi)=\hat{v}\}.
\end{equation*}
By $\eqref{17}$ and $\eqref{18}$, we know that $\mathcal{V}_{b,A}^{\ast}(0)$ and $\hat{\mathcal{V}}_{b,A}^{\ast}(0)$ all have full Lebesgue measure. Moreover, by a general result of Borosh and Fraenkel \cite{IA} or the mass transference principle of Beresnevich and Velani \cite{VS}, for any $v\in\mathbb{R}_{\geq0}\cup\{+\infty\}$, we have
\begin{equation*}
\dim_{\rm H}(\mathcal{V}_{b,A}(v))=\frac{1}{1+v}.
\end{equation*}
By \cite[Theorem 1.1]{YL} and ($\ref{18}$), we know that
\begin{equation}\label{19}
\dim_{\rm H}(\hat{\mathcal{V}}_{b,A}^{\ast}(\hat{v}))\leq\dim_{\rm H}(\hat{\mathcal{V}}_{b,A}(\hat{v}))\leq\left(\frac{\eta-\hat{v}}{\eta+\hat{v}}\right)^{2}\ {\rm for\ any}\ \hat{v}\in[0,\eta].
\end{equation}
A natural question is whether
\begin{equation*}
\dim_{\rm H}(\hat{\mathcal{V}}_{b,A}^{\ast}(\hat{v}))=\dim_{\rm H}(\hat{\mathcal{V}}_{b,A}(\hat{v}))=\left(\frac{\eta-\hat{v}}{\eta+\hat{v}}\right)^{2}\ {\rm for\ all}\ \hat{v}\in[0,\eta]\ ?
\end{equation*}

Our main results are as follows.

Theorem $\ref{11}$ gives an affirmative answer when $\eta=1$. Theorem $\ref{12}$ gives an negative answer when $\eta>1$.
\begin{theorem}\label{11}
Let $b\geq2$ be an integer. If $\eta=1$, then for any $\hat{v}\in[0,1]$, we have
\begin{equation*}
\dim_{\rm H}(\hat{\mathcal{V}}_{b,A}^{\ast}(\hat{v}))=\dim_{\rm H}(\hat{\mathcal{V}}_{b,A}(\hat{v}))=\left(\frac{1-\hat{v}}{1+\hat{v}}\right)^{2}.
\end{equation*}
\end{theorem}
\begin{theorem}\label{12}
Let $b\geq2$ be an integer. If $\lim\limits_{n\to\infty}\frac{a_{n+1}}{a_{n}}$ exists and $$1<\lim_{n\to\infty}\frac{a_{n+1}}{a_{n}}<+\infty,$$ denote
\begin{equation*}
l_{0}=\max\left\{1,\left\lfloor\frac{\log2-\log(\eta-1)}{\log\eta}\right\rfloor+1\right\}.
\end{equation*}
For any
\begin{equation*}
\hat{v}\in\bigcup_{l=l_{0}}^{\infty}\left( \eta-\frac{2\eta}{\eta^{l}+1},\eta-\frac{2}{\eta^{l}}\right),
\end{equation*}
we have
\begin{equation*}
\begin{aligned}
\dim_{\rm H}(\hat{\mathcal{V}}_{b,A}(\hat{v}))&\leq\max\left\{\frac{\eta^{l_{1}+1}-1-\eta^{l_{1}}\hat{v}}{(\eta^{l_{1}+1}-1)(1+\eta^{l_{1}}\hat{v}) },\frac{\eta^{l_{1}}-1-\eta^{l_{1}-1}\hat{v}}{\eta^{l_{1}-1}(\eta(\eta^{l_{1}}-1)-\hat{v})}\right\}\\
&<\left(\frac{\eta-\hat{v}}{\eta+\hat{v}}\right)^{2},
\end{aligned}
\end{equation*}
where
\begin{equation*}
l_{1}=l_{1}(\hat{v}):=\left\lfloor\frac{\log2-\log(\eta-\hat{v})}{\log\eta}\right\rfloor+1.
\end{equation*}
\end{theorem}
The following theorem gives a lower bound of $\dim_{\rm H}(\hat{\mathcal{V}}_{b,A}^{\ast}(\hat{v}))$ when $\lim\limits_{n\to\infty}\frac{a_{n+1}}{a_{n}}$ exists and $1<\lim\limits_{n\to\infty}\frac{a_{n+1}}{a_{n}}<+\infty$.
\begin{theorem}\label{13}
Let $b\geq2$ be an integer. If $\lim\limits_{n\to\infty}\frac{a_{n+1}}{a_{n}}$ exists and $$1<\lim_{n\to\infty}\frac{a_{n+1}}{a_{n}}<+\infty,$$ then for any $\hat{v}\in(0,\eta)$, we have
\begin{equation*}
\dim_{\rm H}(\hat{\mathcal{V}}_{b,A}^{\ast}(\hat{v}))\geq\frac{\eta^{\tilde{l}+1}-1-\eta^{\tilde{l}}\hat{v}}{(\eta^{\tilde{l}+1}-1)(1+\eta^{\tilde{l}}\hat{v})},
\end{equation*}
where
\begin{equation*}
\tilde{l}=\tilde{l}(\hat{v}):=\max\left\{1,\left\lfloor\frac{\log(\eta+1)-\log(\eta-\hat{v})}{\log\eta}\right\rfloor\right\}.
\end{equation*}
\end{theorem}
\begin{remark}
By $\eqref{19}$ and Theorem $\ref{13}$, we know that if $\hat{v}=\eta-\frac{2}{\eta^{l}}$ for some postive integer $l$, then
\begin{equation*}
\dim_{\rm H}(\hat{\mathcal{V}}_{b,A}^{\ast}(\hat{v}))=\dim_{\rm H}(\hat{\mathcal{V}}_{b,A}(\hat{v}))=\left(\frac{\eta-\hat{v}}{\eta+\hat{v}}\right)^{2}.
\end{equation*}
\end{remark}
Theorems $\ref{11}$, $\ref{12}$ and $\ref{13}$ follow from the following two theorems, in which the values of both functions $v_{b,A}$ and $\hat{v}_{b,A}$ are prescribed.
\begin{theorem}\label{14}
Let $b\geq2$ be an integer. Let $\theta$ and $\hat{v}$ be positive real numbers with $\hat{v}<\eta$. If $\theta<\max\left\{1,\frac{1}{\eta-\hat{v}}\right\}$, then
\begin{equation*}
\hat{\mathcal{V}}_{b,A}(\hat{v})\cap\mathcal{V}_{b,A}^{\ast}(\theta\hat{v})=\emptyset.
\end{equation*}
If $\theta\geq\max\left\{1,\frac{1}{\eta-\hat{v}}\right\}$, then
\begin{equation*}
\dim_{\rm H}(\hat{\mathcal{V}}_{b,A}(\hat{v})\cap\mathcal{V}_{b,A}^{\ast}(\theta\hat{v}))\leq\frac{\eta\theta-1-\theta\hat{v}}{(\eta\theta-1)(1+\theta\hat{v})}.
\end{equation*}
Moreover, for every $\theta\geq\max\left\{1,\frac{1}{\eta-\hat{v}}\right\}$ and all $\rho\geq0$, we have
\begin{equation*}
\dim_{\rm H}(\{\xi\in\mathbb{R}:\hat{v}_{b,A}(\xi)\geq\hat{v}\ {\rm and}\ \theta\hat{v}\leq v_{b,A}(\xi)\leq\theta\hat{v}+\rho\})
\leq\frac{\hat{v}(\eta\theta-1-\theta\hat{v})+\rho(\eta-\hat{v})}{(1+\theta\hat{v})((\eta\theta-1)\hat{v}+\eta\rho)}.
\end{equation*}
Furthermore,
\begin{equation*}
\dim_{\rm H}(\hat{\mathcal{V}}_{b,A}(\eta))=0.
\end{equation*}
\end{theorem}

\begin{remark}
In fact, the conclusion $\dim_{\rm H}(\hat{\mathcal{V}}_{b,A}(\eta))=0$ already follows from $$\dim_{\rm H}(\hat{\mathcal{V}}_{b,A}(\hat{v}))\leq\left(\frac{\eta-\hat{v}}{\eta+\hat{v}}\right)^{2}.$$ For completeness, we still state it here.
\end{remark}

\begin{theorem}\label{15}
Let $b\geq2$ be an integer. If $\eta=1$, then for any $\hat{v}\in(0,1)$ and $\theta\geq\frac{1}{1-\hat{v}}$, we have
\begin{equation*}
\dim_{\rm H}(\hat{\mathcal{V}}_{b,A}^{\ast}(\hat{v})\cap\mathcal{V}_{b,A}^{\ast}(\theta\hat{v}))=\dim_{\rm H}(\hat{\mathcal{V}}_{b,A}(\hat{v})\cap\mathcal{V}_{b,A}^{\ast}(\theta\hat{v}))=\frac{\theta-1-\theta\hat{v}}{(\theta-1)(1+\theta\hat{v})}.
\end{equation*}
\end{theorem}
In view of Theorems $\ref{12}$ and $\ref{13}$, we can obtain the following corollary.
\begin{corollary}\label{16}
Let $b\geq2$ be an integer. If $\lim\limits_{n\to\infty}\frac{a_{n+1}}{a_{n}}$ exists and $$1<\lim_{n\to\infty}\frac{a_{n+1}}{a_{n}}<+\infty,$$ denote
\begin{equation*}
l_{0}=\max\left\{1,\left\lfloor\frac{\log2-\log(\eta-1)}{\log\eta}\right\rfloor+1\right\}.
\end{equation*}
For any
\begin{equation*}
\hat{v}\in\bigcup_{l=l_{0}}^{\infty}\left[\eta-\frac{\eta^{l}+\eta^{l+1}-1}{\eta^{2l}},\eta-\frac{2}{\eta^{l}}\right],
\end{equation*}
we have
\begin{equation*}
\dim_{\rm H}(\hat{\mathcal{V}}_{b,A}^{\ast}(\hat{v}))=\dim_{\rm H}(\hat{\mathcal{V}}_{b,A}(\hat{v}))=\frac{\eta^{\tilde{l}+1}-1-\eta^{\tilde{l}}\hat{v}}{(\eta^{\tilde{l}+1}-1)(1+\eta^{\tilde{l}}\hat{v})},
\end{equation*}
where
\begin{equation*}
\tilde{l}=\tilde{l}(\hat{v}):=\max\left\{1,\left\lfloor\frac{\log(\eta+1)-\log(\eta-\hat{v})}{\log\eta}\right\rfloor\right\}.
\end{equation*}
\end{corollary}
The remainder of this paper is organized as follows. We give the definitions of Hausdorff measure and dimension in $\S\ref{20}$. Theorems $\ref{14}$ and $\ref{15}$ are proved in $\S\ref{30}$. In $\S\ref{40}$, we give the proof of Theorem $\ref{11}$. Theorems $\ref{12}$ and $\ref{13}$ are established in $\S\ref{50}$. $\S\ref{60}$ is dedicated to the proof of Corollary $\ref{16}$. Finally, we give two natural examples of sequence $A$ in $\S\ref{70}$, where we apply Theorems $\ref{12}$ and $\ref{13}$ to get (numerical) lower and upper bounds.

Throughout this paper, we denote by $|I|$ the length of the interval $I$ and we use \# to denote the cardinality of a finite set.

\section{Hausdorff measure and dimension}\label{20}
Below we give a brief introduction to Hausdorff measure and dimension. For further details, see $\cite{KJF}$.

Given $m\in\mathbb{N}$, let $F$ be s subset of Euclidean space $\mathbb{R}^{m}$. For $\delta>0$, we say a countable collection of sets $\{U_{i}\}$ in $\mathbb{R}^{m}$ is a $\delta$-cover of $F$ if $F\subset\bigcup_{i}U_{i}$ and their diameters are not greater than $\delta$. Given a dimension function $f$, that is, an increasing, continuous function:$\mathbb{R}_{>0}\to\mathbb{R}_{>0}$ such that $f(x)\to0$ as $x\to0$, the $f$-dimensional Hausdorff measure of $F$ is
\begin{equation*}
\mathcal{H}^{f}(F)=\lim_{\delta\to0}\mathcal{H}^{f}_{\delta}(F),
\end{equation*}
where
\begin{equation*}
\mathcal{H}^{f}_{\delta}(F)=\inf\left\{\sum_{i}f(|U_{i}|):\{U_{i}\}\ {\rm is\ a}\  \delta{\rm -cover\ of}\ F\right\}
\end{equation*}
and $|U_{i}|$ means the diameter of $U_{i}$.

When $f(x)=x^{s}$ for some $s\geq0$, we write $\mathcal{H}^{f}$ as $\mathcal{H}^{s}$. The Hausdorff dimension of a set $F$ is
\begin{equation*}
\dim_{\rm H}F=\inf\{s:\mathcal{H}^{s}(F)=0\}.
\end{equation*}

\section{Proof of Theorems $\ref{14}$ and $\ref{15}$}\label{30}
In order to prove Theorem $\ref{14}$, we need the following lemma.
\begin{lemma}\label{21}
For each $\xi\in\mathbb{R}\setminus\mathbb{Q}$, we have
\begin{equation*}
v_{b,A}(\xi)=+\infty\ {\rm when}\ \hat{v}_{b,A}(\xi)=\eta
\end{equation*}
and
\begin{equation*}
v_{b,A}(\xi)\geq\frac{\hat{v}_{b,A}(\xi)}{\eta-\hat{v}_{b,A}(\xi)}\ {\rm when}\ \hat{v}_{b,A}(\xi)<\eta.
\end{equation*}
\end{lemma}
\begin{proof}
We can assume that $\hat{v}_{b,A}(\xi)>0$, otherwise, the above inequality is trivial. Let
\begin{equation*}
\xi:=\lfloor\xi\rfloor+\sum_{j=1}^{\infty}\frac{x_{j}}{b^{j}}
\end{equation*}
denote the $b$-ary expansion of $\xi$. It is understood that the digits $x_{1},x_{2},\ldots$ all belong to the set $\{0,1,\ldots,b-1\}$. Let
\begin{equation*}
J:=\{j\in\mathbb{N}:x_{a_{j}+1}\in\{0,b-1\}\}.
\end{equation*}
Since $v_{b,A}(\xi)\geq\hat{v}_{b,A}(\xi)>0$, we know that $J$ is an infinite set. So we can denote $J=(j_{k})_{k=1}^{\infty}$, where $j_{1}<j_{2}< \ldots$. For each $j\in J$, because $\xi$ is an irrational real number, there exists $m_{j}'\geq a_{j}+2$, such that either
\begin{equation*}
x_{a_{j}+1}=\ldots=x_{m_{j}'-1}=0,x_{m_{j}'}>0
\end{equation*}
or
\begin{equation*}
x_{a_{j}+1}=\ldots=x_{m_{j}'-1}=b-1,x_{m_{j}'}<b-1.
\end{equation*}
Since $v_{b,A}(\xi)$ is positive, we obtain that
\begin{equation}\label{22}
\limsup_{j\in J,j\to\infty}(m_{j}'-a_{j})=+\infty.
\end{equation}
Now, we take the maximal subsequences $(m_{i_{k}}')_{k=1}^{\infty}$ and $(a_{i_{k}})_{k=1}^{\infty}$ of $(m_{j}')_{j\in J}$ and $(a_{j})_{j\in J}$, respectively, in such a way that the sequence $(m_{i_{k}}'-a_{i_{k}})_{k=1}^{\infty}$ is strictly increasing. More exactly, take $i_{1}=j_{1}$, for any $k\geq1$ such that $i_{k}$ have been defined, let
\begin{equation*}
i_{k+1}:=\min\{j\in J: j>i_{k}\ {\rm and}\ m_{j}'-a_{j}>m_{i_{k}}'-a_{i_{k}}\}.
\end{equation*}
In view of ($\ref{22}$), the sequence $(i_{k})_{k=1}^{\infty}$ is well defined. Furthermore, $m_{i_{k}}'-a_{i_{k}}$ tends to infinity as $k$ tends to infinity. For simplicity, we denote $m_{i_{k}}'$ by $m_{k}$. Note that
\begin{equation*}
b^{-(m_{j}'-a_{j})}<\|b^{a_{j}}\xi\|<b^{-(m_{j}'-a_{j})+1}.
\end{equation*}
By the construction, we have
\begin{equation}\label{23}
v_{b,A}(\xi)=\limsup_{j\in J,j\to\infty}\frac{-\log\|b^{a_{j}}\xi\|}{a_{j}\log b}=\limsup_{j\in J,j\to\infty}\frac{m_{j}'-a_{j}}{a_{j}}
=\limsup_{k\to\infty}\frac{m_{k}-a_{i_{k}}}{a_{i_{k}}}
\end{equation}
and
\begin{equation}\label{24}
\begin{aligned}
&\quad\hat{v}_{b,A}(\xi)=\liminf_{k\to\infty}\frac{-\log\|b^{a_{i_{k}}}\xi\|}{a_{i_{k+1}-1}\log b}\\&=\liminf_{k\to\infty}\frac{m_{k}-a_{i_{k}}}{a_{i_{k+1}-1}}\leq\eta\cdot\liminf_{k\to\infty}\frac{m_{k}-a_{i_{k}}}{a_{i_{k+1}}}\\
&\leq\eta\cdot\liminf_{k\to\infty}\frac{m_{k}-a_{i_{k}}}{m_{k}}=\eta\cdot\left(1-\limsup_{k\to\infty}\frac{a_{i_{k}}}{m_{k}}\right).
\end{aligned}
\end{equation}
The first inequality of ($\ref{24}$) is due to
\begin{equation*}
\limsup_{k\to\infty}\frac{a_{i_{k+1}}}{a_{i_{k+1}-1}}\leq\eta
\end{equation*}
and the second inequality of ($\ref{24}$) follows from $a_{i_{k+1}}\geq m_{k}-1$. When $\hat{v}_{b,A}(\xi)=\eta$, by ($\ref{24}$), we have
\begin{equation*}
\limsup_{k\to\infty}\frac{a_{i_{k}}}{m_{k}}=0.
\end{equation*}
It then follows from ($\ref{23}$) that
\begin{equation*}
v_{b,A}(\xi)=+\infty.
\end{equation*}
When $\hat{v}_{b,A}(\xi)<\eta$, we consider two cases. If $v_{b,A}(\xi)=+\infty$, obviously,
\begin{equation*}
v_{b,A}(\xi)\geq\frac{\hat{v}_{b,A}(\xi)}{\eta-\hat{v}_{b,A}(\xi)}.
\end{equation*}
If $v_{b,A}(\xi)<+\infty$, by ($\ref{23}$), we have
\begin{equation*}
\limsup_{k\to\infty}\frac{m_{k}}{a_{i_{k}}}<+\infty.
\end{equation*}
Therefore,
\begin{equation*}
\limsup_{k\to\infty}\frac{a_{i_{k}}}{m_{k}}\geq\liminf_{k\to\infty}\frac{a_{i_{k}}}{m_{k}}>0.
\end{equation*}
Thus,
\begin{equation*}
\limsup_{k\to\infty}\frac{a_{i_{k}}}{m_{k}}\cdot\limsup_{k\to\infty}\frac{m_{k}}{a_{i_{k}}}\geq1.
\end{equation*}
In view of ($\ref{23}$) and ($\ref{24}$), we obtain that
\begin{equation*}
(1-\eta^{-1}\hat{v}_{b,A}(\xi))(v_{b,A}(\xi)+1)\geq1.
\end{equation*}
Hence,
\begin{equation*}
v_{b,A}(\xi)\geq\frac{\hat{v}_{b,A}(\xi)}{\eta-\hat{v}_{b,A}(\xi)}.
\end{equation*}
\textrm{}
\end{proof}
Now, we begin the proof the Theorem $\ref{14}$.
\begin{proof}[Proof of Theorem $\ref{14}$]
If $\theta<1$, then for any $\xi\in\hat{\mathcal{V}}_{b,A}(\hat{v})$, we have
\begin{equation*}
v_{b,A}(\xi)\geq\hat{v}_{b,A}(\xi)\geq\hat{v}>\theta\hat{v}.
\end{equation*}
Thus, $\xi\notin\mathcal{V}_{b,A}^{\ast}(\theta\hat{v})$, which implies that
\begin{equation*}
\hat{\mathcal{V}}_{b,A}(\hat{v})\cap\mathcal{V}_{b,A}^{\ast}(\theta\hat{v})=\emptyset.
\end{equation*}
If $\theta<\frac{1}{\eta-\hat{v}}$, then for all $\xi\in\mathcal{V}_{b,A}^{\ast}(\theta\hat{v})$, we have $\xi\in\mathbb{R}\setminus\mathbb{Q}$ and
\begin{equation}\label{25}
v_{b,A}(\xi)=\theta\hat{v}<\frac{\hat{v}}{\eta-\hat{v}}.
\end{equation}
In view of Lemma $\ref{21}$, we have $\hat{v}_{b,A}(\xi)<\eta$ and
\begin{equation}\label{26}
v_{b,A}(\xi)\geq\frac{\hat{v}_{b,A}(\xi)}{\eta-\hat{v}_{b,A}(\xi)}.
\end{equation}
By ($\ref{25}$) and ($\ref{26}$), we obtain that $\hat{v}_{b,A}(\xi)<\hat{v}$. Hence,
\begin{equation*}
\hat{\mathcal{V}}_{b,A}(\hat{v})\cap\mathcal{V}_{b,A}^{\ast}(\theta\hat{v})=\emptyset.
\end{equation*}
Now let $\theta\geq\max\left\{1,\frac{1}{\eta-\hat{v}}\right\}$. We can assume that $\hat{\mathcal{V}}_{b,A}(\hat{v})\cap\mathcal{V}_{b,A}^{\ast}(\theta\hat{v})\neq\emptyset$ as otherwise the inequality is trivial. For each $\xi\in\hat{\mathcal{V}}_{b,A}(\hat{v})\cap\mathcal{V}_{b,A}^{\ast}(\theta\hat{v})$, we have
\begin{equation}\label{27}
\hat{v}_{b,A}(\xi)\geq\hat{v}\ {\rm and}\ v_{b,A}(\xi)=\theta\hat{v}.
\end{equation}
Therefore, $\xi\in\mathbb{R}\setminus\mathbb{Q}$. Let
\begin{equation*}
\xi:=\lfloor\xi\rfloor+\sum_{j=1}^{\infty}\frac{x_{j}}{b^{j}}
\end{equation*}
denote the $b$-ary expansion of $\xi$. Denote by
\begin{equation*}
J:=\{j\in\mathbb{N}:x_{a_{j}+1}\in\{0,b-1\}\}.
\end{equation*}
By the proof of Lemma $\ref{21}$, we know that there exists subsequences $(m_{k})_{k=1}^{\infty}$ and $(a_{i_{k}})_{k=1}^{\infty}$ of $(m_{j}')_{j\in J}$ and $(a_{j})_{j\in J}$, respectively, such that for any $k\in\mathbb{N}$, we have either
\begin{equation*}
x_{a_{i_{k}}+1}=\ldots=x_{m_{k}-1}=0,x_{m_{k}}>0
\end{equation*}
or
\begin{equation*}
x_{a_{i_{k}}+1}=\ldots=x_{m_{k}-1}=b-1,x_{m_{k}}<b-1.
\end{equation*}
Furthermore, we have $m_{k+1}-a_{i_{k+1}}>m_{k}-a_{i_{k}}$ for every $k\in\mathbb{N}$. What is more,
\begin{equation}\label{28}
v_{b,A}(\xi)=\limsup_{k\to\infty}\frac{m_{k}-a_{i_{k}}}{a_{i_{k}}}
\end{equation}
and
\begin{equation}\label{29}
\hat{v}_{b,A}(\xi)=\liminf_{k\to\infty}\frac{m_{k}-a_{i_{k}}}{a_{i_{k+1}-1}}\leq\eta\cdot\liminf_{k\to\infty}\frac{m_{k}-a_{i_{k}}}{a_{i_{k+1}}}.
\end{equation}
Take a subsequence $(k_{j})_{j=1}^{\infty}$ along which the supremum of ($\ref{28}$) is obtained. For simplicity, we still write $(m_{k})_{k=1}^{\infty}$, $(a_{i_{k}})_{k=1}^{\infty}$ for the subsequence $(m_{k_{j}})_{j=1}^{\infty}$ and $(a_{i_{k_{j}}})_{j=1}^{\infty}$. In view of ($\ref{27}$), ($\ref{28}$) and ($\ref{29}$), we have
\begin{equation}\label{211}
\lim_{k\to\infty}\frac{m_{k}-a_{i_{k}}}{a_{i_{k}}}=\theta\hat{v}
\end{equation}
and
\begin{equation}\label{212}
\liminf_{k\to\infty}\frac{m_{k}-a_{i_{k}}}{a_{i_{k+1}}}\geq\eta^{-1}\hat{v}.
\end{equation}
For any $\varepsilon\in(0,\frac{1}{2}\eta^{-1}\hat{v})$, due to ($\ref{211}$), ($\ref{212}$) and $\lim_{n\to\infty}a_{n}=\infty$, there exists $K_{1}\in\mathbb{N}$, such that for each $k\geq K_{1}$, we have
\begin{equation}\label{213}
\theta\hat{v}-\frac{1}{2}\varepsilon<\frac{m_{k}-a_{i_{k}}}{a_{i_{k}}}<\theta\hat{v}+\frac{1}{2}\varepsilon,
\end{equation}
\begin{equation}\label{214}
\frac{m_{k}-a_{i_{k}}}{a_{i_{k+1}}}>\eta^{-1}\hat{v}-\frac{1}{2}\varepsilon
\end{equation}
and
\begin{equation}\label{215}
a_{i_{k+1}}\geq\frac{2}{\varepsilon}.
\end{equation}
Combining the second inequality of ($\ref{213}$), ($\ref{214}$) and ($\ref{215}$), we have
\begin{equation*}
m_{k}-a_{i_{k}}-1>\frac{\eta^{-1}\hat{v}-\varepsilon}{\theta\hat{v}+\frac{1}{2}\varepsilon}\cdot (m_{k+1}-a_{i_{k+1}}-1).
\end{equation*}
Since $\theta\geq\max\left\{1,\frac{1}{\eta-\hat{v}}\right\}>\eta^{-1}$ and $\varepsilon\in(0,\frac{1}{2}\eta^{-1}\hat{v})$, we have
\begin{equation}\label{224}
0<\frac{\eta^{-1}\hat{v}-\varepsilon}{\theta\hat{v}+\frac{1}{2}\varepsilon}<1.
\end{equation}
So we can fix $K_{2}>K_{1}$ satisfying
\begin{equation}\label{225}
\left(\frac{\eta^{-1}\hat{v}-\varepsilon}{\theta\hat{v}+\frac{1}{2}\varepsilon}\right)^{K_{2}-K_{1}}<\varepsilon.
\end{equation}
For all $k\geq K_{2}$, the sum of all the lengths of the blocks of 0 or $b-1$ in the prefix of length $a_{i_{k}}$ of the infinite sequence $x_{1}x_{2}\ldots$ is, at least equal to
\begin{equation}\label{223}
\begin{aligned}
\sum_{l=K_{1}}^{k-1}(m_{l}-a_{i_{l}}-1)&\geq\sum_{l=K_{1}}^{k-1}(m_{k-1}-a_{i_{k-1}}-1)\cdot\left(\frac{\eta^{-1}\hat{v}-\varepsilon}{\theta\hat{v}+\frac{1}{2}\varepsilon}\right)^{k-l-1}\\
&>(\eta^{-1}\hat{v}-\varepsilon)a_{i_{k}}\sum_{l=K_{1}}^{k-1}\left(\frac{\eta^{-1}\hat{v}-\varepsilon}{\theta\hat{v}+\frac{1}{2}\varepsilon}\right)^{k-l-1}\\ &=(\eta^{-1}\hat{v}-\varepsilon)a_{i_{k}}\frac{\theta\hat{v}+\frac{1}{2}\varepsilon}{(\theta-\eta^{-1})\hat{v}+\frac{3}{2}\varepsilon}\left(1-\left(\frac{\eta^{-1}\hat{v}-\varepsilon}{\theta\hat{v}+\frac{1}{2}\varepsilon}\right)^{ k-K_{1}}\right) \\
&>a_{i_{k}}\left(\frac{\theta\hat{v}}{\eta\theta-1}-\varepsilon'\right),
\end{aligned}
\end{equation}
where
\begin{equation*}
\varepsilon'=\left(\frac{\theta\hat{v}}{\eta\theta-1}-\frac{(\hat{v}-\eta\varepsilon)(\theta\hat{v}+\frac{1}{2}\varepsilon)}{(\eta\theta-1)\hat{v}+\frac{3}{2}\eta\varepsilon}\right)+\varepsilon\cdot\frac{(\hat{v}-\eta\varepsilon)(\theta\hat{v}+\frac{1}{2}\varepsilon)}{(\eta\theta-1)\hat{v}+\frac{3}{2}\eta\varepsilon}. \end{equation*}
The last inequality of $\eqref{223}$ is due to $\eqref{224}$, $\eqref{225}$ and $k\geq K_{2}$.
Note that $\varepsilon'>0$ and $\varepsilon'$ tends to zero as $\varepsilon$ tends to zero. By the first inequality of ($\ref{213}$), we have
\begin{equation*}
\begin{aligned}
m_{k}>\left(1+\theta\hat{v}-\frac{1}{2}\varepsilon\right)\cdot a_{i_{k}}&\geq\left(1+\theta\hat{v}-\frac{1}{2}\varepsilon\right)\cdot(m_{k-1}-1) \\
&\geq(1+\theta\hat{v}-\varepsilon)\cdot m_{k-1}\\ &>\left(1+\frac{1}{2}\eta^{-1}\hat{v}\right)\cdot m_{k-1}
\end{aligned}
\end{equation*}
for $k$ large enough. The last inequality is due to $\varepsilon\in(0,\frac{1}{2}\eta^{-1}\hat{v})$ and $\theta>\frac{1}{\eta-\hat{v}}$. Thus, $(m_{k})_{k=1}^{\infty}$ increases at least exponentially. Since $a_{i_{k}}\geq m_{k-1}-1$ for $k\geq2$, the sequence $(a_{i_{k}})_{k=1}^{\infty}$ also increases at least exponentially. Consequently, there exists $K_{3}\in\mathbb{N}$, for all $k\geq K_{3}$, we have
\begin{equation*}
k\leq c\cdot\log a_{i_{k}},\ {\rm where}\ c:=\left(\log\left(1+\frac{1}{4}\eta^{-1}\hat{v}\right)\right)^{-1}.
\end{equation*}
Let $K_{4}=\max\{K_{2},K_{3}\}$. For each $k\geq K_{4}$, we have
\begin{equation}\label{216}
k\leq c\cdot\log a_{i_{k}}\ {\rm and}\ \sum_{l=K_{1}}^{k-1}(m_{l}-a_{i_{l}}-1)>a_{i_{k}}\left(\frac{\theta\hat{v}}{\eta\theta-1}-\varepsilon'\right).
\end{equation}
For any $\delta>0$, let us construct a $\delta-$cover of $\hat{\mathcal{V}}_{b,A}(\hat{v})\cap\mathcal{V}_{b,A}^{\ast}(\theta\hat{v})\cap[0,1)$. Denote
\begin{equation*}
N_{\delta,\varepsilon}:=\max\left\{\left\lfloor\frac{\log\delta^{-1}}{\log b}\right\rfloor,\left\lfloor\frac{1}{\theta\hat{v}-\frac{1}{2}\varepsilon}\right\rfloor\right\}+1.
\end{equation*}
For any $n\in\mathbb{N}$ and $x_{1},\ldots,x_{n}$ in $\{0,1,\ldots,b-1\}$, denote by $I_{n}(x_{1},\ldots,x_{n})$ the interval consist of the real numbers in $[0,1)$ whose $b$-ary expansion starts with $x_{1}\ldots x_{n}$. For each $N\geq N_{\delta,\varepsilon}$, let $\Gamma_{N,\varepsilon}$ be the set of all $(x_{1},\ldots, x_{\lfloor(1+\theta\hat{v}-\frac{1}{2}\varepsilon)N\rfloor})$ satisfying $x_{N+1}=\ldots=x_{\lfloor(1+\theta\hat{v}-\frac{1}{2}\varepsilon)N\rfloor}=0$ or $b-1$ and there exists $n$ blocks of 0 or $b-1$ of $x_{1}\ldots x_{N}$, with $n\leq c\log N$ and the total length of these $n$ blocks is at least equal to $N\left(\frac{\theta\hat{v}}{\eta\theta-1}-\varepsilon'\right)$. For $L\geq N_{\delta,\varepsilon}$, denote
\begin{equation*}
A_{\delta,\varepsilon,L}:=\bigcup_{N=L}^{\infty}\bigcup_{(x_{1},\ldots, x_{\lfloor(1+\theta\hat{v}-\frac{1}{2}\varepsilon)N\rfloor}) \in\Gamma_{N,\varepsilon}}I_{\lfloor(1+\theta\hat{v}-\frac{1}{2}\varepsilon)N\rfloor}( x_{1}, \ldots,x_{\lfloor(1+\theta\hat{v}-\frac{1}{2}\varepsilon)N\rfloor}).
\end{equation*}
For any $\xi\in\hat{\mathcal{V}}_{b,A}(\hat{v})\cap\mathcal{V}_{b,A}^{\ast}(\theta\hat{v})\cap[0,1)$, for all $k\geq K_{4}$ and $a_{i_{k}}\geq L$, in view of ($\ref{216}$), we know that
\begin{equation*}
k\leq c\cdot\log a_{i_{k}}\ {\rm and}\ \sum_{l=K_{1}}^{k-1}(m_{l}-a_{i_{l}}-1)>a_{i_{k}}\left(\frac{\theta\hat{v}}{\eta\theta-1}-\varepsilon'\right).
\end{equation*}
By the definition of $(a_{i_{k}})_{k\geq1}$ and the first inequality of $\eqref{213}$, we have
\begin{equation*}
x_{a_{i_{k}}+1}=\ldots=x_{\lfloor(1+\theta\hat{v}-\frac{1}{2}\varepsilon)a_{i_{k}}\rfloor}=0\ {\rm or}\ b-1.
\end{equation*}
Thus,
\begin{equation*}
\xi\in\bigcup_{(x_{1},\ldots,x_{\lfloor(1+\theta\hat{v}-\frac{1}{2}\varepsilon)a_{i_{k}}\rfloor}) \in\Gamma_{a_{i_{k}},\varepsilon}}I_{\lfloor(1+\theta\hat{v}-\frac{1}{2}\varepsilon)a_{i_{k}}\rfloor}(x_{1},\ldots, x_{\lfloor(1+\theta\hat{v}-\frac{1}{2}\varepsilon)a_{i_{k}}\rfloor})\subset A_{\delta,\varepsilon,L}.
\end{equation*}
Therefore,
\begin{equation*}
\hat{\mathcal{V}}_{b,A}(\hat{v})\cap\mathcal{V}_{b,A}^{\ast}(\theta\hat{v})\cap[0,1)\subset A_{\delta,\varepsilon,L}.
\end{equation*}
On the other hand, for any $N\geq L$, we have
\begin{equation*}
|I_{\lfloor(1+\theta\hat{v}-\frac{1}{2}\varepsilon)N\rfloor}(x_{1},\ldots,x_{\lfloor(1+\theta\hat{v}-\frac{1}{2}\varepsilon)N\rfloor})|\doteq b^{-\lfloor(1+\theta\hat{v}-\frac{1}{2}\varepsilon)N\rfloor}<b^{-N}<\delta.
\end{equation*}
Hence, $A_{\delta,\varepsilon,L}$ is a $\delta-$cover of $\hat{\mathcal{V}}_{b,A}(\hat{v})\cap\mathcal{V}_{b,A}^{\ast}(\theta\hat{v})\cap[0,1)$.
Thus, for any $s\geq0$, we have
\begin{equation*}
\begin{aligned}
&\quad\mathcal{H}_{\delta}^{s}(\hat{\mathcal{V}}_{b,A}(\hat{v})\cap\mathcal{V}_{b,A}^{\ast}(\theta\hat{v})\cap[0,1))\\
&\leq\sum_{N=L}^{\infty}\sum_{(x_{1},\ldots,x_{\lfloor(1+\theta\hat{v}-\frac{1}{2}\varepsilon)N\rfloor}) \in\Gamma_{N,\varepsilon}}b^{-\lfloor(1+\theta\hat{v}-\frac{1}{2}\varepsilon)N\rfloor s}\\ &=\sum_{N= L}^{\infty}b^{-\lfloor(1+\theta\hat{v}-\frac{1}{2}\varepsilon)N\rfloor s}\cdot\sharp\Gamma_{N,\varepsilon}.
\end{aligned}
\end{equation*}
So, we need to estimate $\#\Gamma_{N,\varepsilon}$. Note that for any $n$ with $1\leq n\leq c\log N$ and every $M$ with $\left\lfloor N\left(\frac{\theta\hat{v}}{\eta\theta-1}-\varepsilon'\right)\right\rfloor\leq M\leq N$, there are at most
\begin{equation*}
2^{n}N^{n}\binom{n-1}{M-1}
\end{equation*}
choices of the $n$ blocks of 0 or $b-1$ of $x_{1}\ldots x_{N}$, with the total length of these $n$ blocks are equal to $M$. Hence,
\begin{equation*}
\begin{aligned}
\# \Gamma_{N,\varepsilon}&\leq\sum_{n=1}^{\lfloor c\log N\rfloor}\left(2^{n+1}N^{n}\sum_{M=\left\lfloor N\left(\frac{\theta\hat{v}}{\eta\theta-1}-\varepsilon'\right)\right\rfloor}^{N}\binom{n-1}{M-1}b^{N-M}\right) \\
&\leq2b\cdot(2N^{2})^{c\log N}c\log N b^{N\left(\frac{\eta\theta-1-\theta\hat{v}}{\eta\theta-1}+\varepsilon'\right)}.
\end{aligned}
\end{equation*}
Thus,
\begin{equation*}
\begin{aligned}
&\quad\mathcal{H}_{\delta}^{s}(\hat{\mathcal{V}}_{b,A}(\hat{v})\cap\mathcal{V}_{b,A}^{\ast}(\theta\hat{v})\cap[0,1))\\
&\leq2b^{1+s}\sum_{N=L}^{\infty}(2N^{2})^{c\log N}c\log N\cdot b^{\left(-(1+\theta\hat{v}-\frac{1}{2}\varepsilon)s+\left(\frac{\eta\theta-1-\theta\hat{v}}{\eta\theta-1}+\varepsilon'\right)\right)N}.
\end{aligned}
\end{equation*}
Therefore, for any $s>(1+\theta\hat{v}-\frac{1}{2}\varepsilon)^{-1}\left(\frac{\eta\theta-1-\theta\hat{v}}{\eta\theta-1}+\varepsilon'\right)$, letting $L\to+\infty$, we have
\begin{equation*}
\mathcal{H}_{\delta}^{s}(\hat{\mathcal{V}}_{b,A}(\hat{v})\cap\mathcal{V}_{b,A}^{\ast}(\theta\hat{v})\cap[0,1))=0.
\end{equation*}
Letting $\delta\to0$, we obtain that
\begin{equation*}
\mathcal{H}^{s}(\hat{\mathcal{V}}_{b,A}(\hat{v})\cap\mathcal{V}_{b,A}^{\ast}(\theta\hat{v})\cap[0,1))=0.
\end{equation*}
Hence, $s\geq\dim_{\rm H}(\hat{\mathcal{V}}_{b,A}(\hat{v})\cap\mathcal{V}_{b,A}^{\ast}(\theta\hat{v})\cap[0,1))$. Therefore,
\begin{equation*}
(1+\theta\hat{v}-\frac{1}{2}\varepsilon)^{-1}\left(\frac{\eta\theta-1-\theta\hat{v}}{\eta\theta-1}+\varepsilon'\right)\geq\dim_{\rm H}(\hat{\mathcal{V}}_{b,A}(\hat{v})\cap\mathcal{V}_{b,A}^{\ast}(\theta\hat{v})\cap[0,1)).
\end{equation*}
Letting $\varepsilon\to0$, we obtain that
\begin{equation*}
\dim_{\rm H}(\hat{\mathcal{V}}_{b,A}(\hat{v})\cap\mathcal{V}_{b,A}^{\ast}(\theta\hat{v})\cap[0,1))\leq\frac{\eta\theta-1-\theta\hat{v}}{(1+\theta\hat{v})(\eta\theta-1)}.
\end{equation*}
By the countable stability of Hausdorff dimension and the definition of $\hat{\mathcal{V}}_{b,A}(\hat{v})$, $\mathcal{V}_{b,A}^{\ast}(\theta\hat{v})$, we have
\begin{equation*}
\dim_{\rm H}(\hat{\mathcal{V}}_{b,A}(\hat{v})\cap\mathcal{V}_{b,A}^{\ast}(\theta\hat{v}))=\dim_{\rm H}(\hat{\mathcal{V}}_{b,A}(\hat{v})\cap\mathcal{V}_{b,A}^{\ast}(\theta\hat{v})\cap[0,1))\leq\frac{\eta\theta-1-\theta\hat{v}}{(1+\theta\hat{v})(\eta\theta-1)}.
\end{equation*}
Actually, we have proved that, for every $\theta\geq\max\left\{1,\frac{1}{\eta-\hat{v}}\right\}$ and all $\rho\geq0$, we have
\begin{equation*}
\dim_{\rm H}(\{\xi\in\mathbb{R}:\hat{v}_{b,A}(\xi)\geq\hat{v}\ {\rm and}\ \theta\hat{v}\leq v_{b,A}(\xi)\leq\theta\hat{v}+\rho\})
\leq\frac{\hat{v}(\eta\theta-1-\theta\hat{v})+\rho(\eta-\hat{v})}{(1+\theta\hat{v})((\eta\theta-1)\hat{v}+\eta\rho)}.
\end{equation*}
\textrm{}
\end{proof}
At the end of this section, we give the proof of Theorem $\ref{15}$.
\begin{proof}[Proof of Theorem $\ref{15}$]
By Theorem $\ref{14}$, we only need to show that
\begin{equation}\label{217}
\dim_{\rm H}(\hat{\mathcal{V}}_{b,A}^{\ast}(\hat{v})\cap\mathcal{V}_{b,A}^{\ast}(\theta\hat{v}))\geq\frac{\theta-1-\theta\hat{v}}{(\theta-1)(1+\theta\hat{v})}
\end{equation}
for every $\theta\geq\frac{1}{1-\hat{v}}$. We can suppose that $\theta>\frac{1}{1-\hat{v}}$, otherwise, the above inequality is trivial. We construct a suitable Cantor type subset of $\hat{\mathcal{V}}_{b,A}^{\ast}(\hat{v})\cap\mathcal{V}_{b,A}^{\ast}(\theta\hat{v})$. Fix $i_{1}\in\mathbb{N}$ such that
\begin{equation*}
a_{i_{1}}>\max\{3(\theta\hat{v})^{-1},(\theta-1)^{-1}(\theta\hat{v})^{-1},(\theta-1-\theta\hat{v})^{-1}\}.
\end{equation*}
Let $m_{1}=\lfloor(1+\theta\hat{v})a_{i_{1}}\rfloor$. Let $k\geq1$ be such that $i_{k}$ have been defined. Then we define
\begin{equation*}
m_{k}=\lfloor(1+\theta\hat{v})a_{i_{k}}\rfloor\ {\rm and}\ i_{k+1}=\min\{j>i_{k}:a_{j}>\theta a_{i_{k}}\}.
\end{equation*}
Since $\lim_{n\to\infty}a_{n}=+\infty$, the sequence $(i_{k})_{k\geq1}$ is well defined. By the above construction, we have
\begin{equation*}
a_{i_{k}}+3\leq m_{k}\leq a_{i_{k+1}}-2,
\end{equation*}
\begin{equation*}
 m_{k+1}-a_{i_{k+1}}>m_{k}-a_{i_{k}},
\end{equation*}
\begin{equation}\label{218}
\lim_{k\to\infty}\frac{m_{k}-a_{i_{k}}}{a_{i_{k}}}=\theta\hat{v}
\end{equation}
and
\begin{equation}\label{219}
\lim_{k\to\infty}\frac{m_{k}-a_{i_{k}}}{a_{i_{k+1}}}=\hat{v}.
\end{equation}
It is worth mentioning that the condition $\eta=1$ is used in $\eqref{219}$.

In order to construct a Cantor type subset of $\hat{\mathcal{V}}_{b,A}^{\ast}(\hat{v})\cap\mathcal{V}_{b,A}^{\ast}(\theta\hat{v})$, we consider two cases.

If $b\geq3$, let $E_{\theta,\hat{v}}$ be the set of all real numbers $\xi$ in $(0,1)$ whose $b$-ary expansion $\xi=\sum_{j\geq1}\frac{x_{j}}{b^{j}}$ satisfies, for each $k\geq1$,
\begin{equation*}
x_{a_{i_{k}}}=1,x_{a_{i_{k}}+1}=\ldots=x_{m_{k}-1}=0,x_{m_{k}}=1,
\end{equation*}
and
\begin{equation*}
x_{m_{k}+(m_{k}-a_{i_{k}})}=x_{m_{k}+2(m_{k}-a_{i_{k}})}=\ldots=x_{m_{k}+t_{k}(m_{k}-a_{i_{k}})}=1,
\end{equation*}
where $t_{k}$ is the largest integer such that $m_{k}+t_{k}(m_{k}-a_{i_{k}})<a_{i_{k+1}}$. Since
\begin{equation*}
t_{k}<\frac{a_{i_{k+1}}-m_{k}}{m_{k}-a_{i_{k}}}<\frac{2}{\hat{v}}
\end{equation*}
for all $k$ large enough, the sequence $(t_{k})_{k\geq1}$ is bounded. We check that the maximal length of blocks of zeros of the finite sequence $x_{a_{i_{k}}}x_{a_{i_{k}}+1}\ldots x_{a_{i_{k+1}}-1}$ is equal to $m_{k}-a_{i_{k}}-1$. Since $b-2\geq1$, we have
\begin{equation*}
\|b^{a_{n}}\xi\|>b^{-(m_{k}-a_{i_{k}})}
\end{equation*}
for every $i_{k}\leq n<i_{k+1}$ with $k\geq1$. Thus,
\begin{equation*}
\hat{v}_{b,A}(\xi)=\hat{v}\ {\rm and}\ v_{b,A}(\xi)=\theta\hat{v}.
\end{equation*}

If $b=2$, let $E_{\theta,\hat{v}}$ be the set of all real numbers $\xi$ in $(0,1)$ whose $b$-ary expansion $\xi=\sum_{j\geq1}\frac{x_{j}}{b^{j}}$ satisfies, for each $k\geq1$,
\begin{equation*}
x_{a_{i_{k}}}=1,x_{a_{i_{k}}+1}=\ldots=x_{m_{k}-1}=0,x_{m_{k}}=1,
\end{equation*}
\begin{equation*}
x_{m_{k}+(m_{k}-a_{i_{k}})}=x_{m_{k}+2(m_{k}-a_{i_{k}})}=\ldots=x_{m_{k}+t_{k}(m_{k}-a_{i_{k}})}=1,
\end{equation*}
and
\begin{equation*}
x_{m_{k}+(m_{k}-a_{i_{k}})-1}=x_{m_{k}+2(m_{k}-a_{i_{k}})-1}=\ldots=x_{m_{k}+t_{k}(m_{k}-a_{i_{k}})-1}=0,
\end{equation*}
where $t_{k}$ is the largest integer such that $m_{k}+t_{k}(m_{k}-a_{i_{k}})<a_{i_{k+1}}$. We can check that
\begin{equation*}
\|b^{a_{n}}\xi\|>b^{-(m_{k}-a_{i_{k}}+2)}
\end{equation*}
for every $i_{k}\leq n<i_{k+1}$ with $k\geq1$. Thus,
\begin{equation*}
\hat{v}_{b,A}(\xi)=\hat{v}\ {\rm and}\ v_{b,A}(\xi)=\theta\hat{v}.
\end{equation*}

The remaining proof for the case $b\geq3$ and the case $b=2$ are similar. For simplicity, we assume that $b\geq3$ and leave the details of the proof for case $b=2$ to the reader. We will now estimate the Hausdorff dimension of $E_{\theta,\hat{v}}$ from below.

Let $n\in\mathbb{N}$. For $x_{1},\ldots,x_{n}\in\{0,1,\ldots,b-1\}$, denote by $I_{n}(x_{1},\ldots,x_{n})$ the interval consist of the real numbers in $[0,1)$ whose $b$-ary expansion starts with $x_{1},\ldots,x_{n}$. Define a Bernoulli measure $\mu$ on $E_{\theta,\hat{v}}$ as follows. We distribute the mass uniformly.

For every $n\geq a_{i_{2}}$, there exists $k\geq2$ such that
\begin{equation*}
a_{i_{k}}\leq n\leq a_{i_{k+1}}.
\end{equation*}
If $a_{i_{k}}\leq n\leq m_{k}$, then define
\begin{equation*}
\mu(I_{n}(x_{1},\ldots,x_{n}))=b^{-\left(a_{i_{1}}-1+\sum_{j=1}^{k-1}(a_{i_{j+1}}-m_{j}-t_{j}-1)\right)}.
\end{equation*}
Observe that we have $\mu(I_{a_{i_{k}}})=\ldots=\mu(I_{m_{k}})$. \\
If $m_{k}<n<a_{i_{k+1}}$, then define
\begin{equation*}
\mu(I_{n}(x_{1},\ldots,x_{n}))=b^{-\left(a_{i_{1}}-1+\sum_{j=1}^{k-1}(a_{i_{j+1}}-m_{j}-t_{j}-1)+n-m_{k}-t\right)}
\end{equation*}
where $t$ is the largest integer such that $m_{k}+t(m_{k}-a_{i_{k}})\leq n$. It is routine to check that $\mu$ is well defined on $E_{\theta,\hat{v}}$.

Now we calculate the lower local dimension of $\mu$ at $x\in E_{\theta,\hat{v}}$, that is,
\begin{equation*}
\underline{\dim}_{\rm loc}\mu(x):=\liminf_{r\to 0}\frac{\log\mu(B(x,r))}{\log r},
\end{equation*}
where $B(x,r)$ stands for the ball centered at $x$ with radius $r$. We will prove that
\begin{equation}\label{220}
\underline{\dim}_{\rm loc}\mu(x)=\frac{\theta-1-\theta\hat{v}}{(\theta-1)(\theta\hat{v}+1)}
\end{equation}
for all $x\in E_{\theta,\hat{v}}$.

Firstly, we show that
\begin{equation*}
\liminf_{n\to\infty}\frac{\log\mu(I_{n})}{\log|I_{n}|}=\frac{\theta-1-\theta\hat{v}}{(\theta-1)(\theta\hat{v}+1)},
\end{equation*}
where $I_{n}$ is the $n$th basic intervals. If $n=m_{k}$, then
\begin{equation*}
\liminf_{k\to\infty}\frac{\log\mu(I_{m_{k}})}{\log|I_{m_{k}}|}=\liminf_{k\to\infty}\frac{a_{i_{1}}-1+\sum_{j=1}^{k-1}(a_{i_{j+1}}-m_{j}-t_{j}-1)}{m_{k}}.
\end{equation*}
Recalling that $(t_{k})_{k\geq1}$ is bounded and that $(m_{k})_{k\geq1}$ grows exponentially fast in terms of $k$, we have
\begin{equation*}
\liminf_{k\to\infty}\frac{\log\mu(I_{m_{k}})}{\log|I_{m_{k}}|}=\liminf_{k\to\infty}\frac{\sum_{j=1}^{k-1}(a_{i_{j+1}}-m_{j})}{m_{k}}.
\end{equation*}
It follows from ($\ref{218}$) and ($\ref{219}$) that
\begin{equation*}
\lim_{k\to\infty}\frac{m_{k}}{a_{i_{k}}}=1+\theta\hat{v},\ \lim_{k\to\infty}\frac{m_{k+1}}{m_{k}}=\theta\ {\rm and}\ \lim_{k\to\infty}\frac{a_{i_{k+1}}}{m_{k}}=\frac{\theta}{1+\theta\hat{v}}.
\end{equation*}
Hence, by the Stolz-Ces\`{a}ro theorem,
\begin{equation*}
\lim_{k\to\infty}\frac{\sum_{j=1}^{k-1}(a_{i_{j+1}}-m_{j})}{m_{k}}=\lim_{k\to\infty}\frac{a_{i_{k+1}}-m_{k}}{m_{k+1}-m_{k}}=\lim_{k\to\infty}\frac{\frac{a_{i_{k+1}}}{m_{k}}-1}{\frac{m_{k+1}}{m_{k}}-1}= \frac{\theta-1-\theta\hat{v}}{(\theta-1)(\theta\hat{v}+1)}.
\end{equation*}
Therefore,
\begin{equation*}
\liminf_{k\to\infty}\frac{\log\mu(I_{m_{k}})}{\log|I_{m_{k}}|}=\frac{\theta-1-\theta\hat{v}}{(\theta-1)(\theta\hat{v}+1)}.
\end{equation*}
For any $\varepsilon>0$, we prove that
\begin{equation*}
\frac{\log\mu(I_{n})}{\log|I_{n}|}>(1-\varepsilon)\frac{\log\mu(I_{m_{k}})}{\log|I_{m_{k}}|}
\end{equation*}
for all $n$th basic intervals with $n$ large enough, where $k\geq2$ satisfying $a_{i_{k}}\leq n<a_{i_{k+1}}$. If $a_{i_{k}}\leq n\leq m_{k}$, then
\begin{equation*}
\frac{\log\mu(I_{n})}{\log|I_{n}|}\geq\frac{\log\mu(I_{m_{k}})}{\log|I_{m_{k}}|}.
\end{equation*}
If $m_{k}<n<a_{i_{k+1}}$, write $n=m_{k}+t(m_{k}-a_{i_{k}})+l$, where $t,l$ are integers with $0\leq t\leq t_{k}$ and $0\leq l<m_{k}-a_{i_{k}}$. Then we have
\begin{equation*}
\mu(I_{n})=\mu(I_{m_{k}})\cdot b^{-(t(m_{k}-a_{i_{k}}-1)+l)}\ {\rm and}\ |I_{n}|=|I_{m_{k}}|\cdot b^{-(t(m_{k}-a_{i_{k}})+l)}.
\end{equation*}
Since $0\leq t\leq t_{k}$ and $(t_{k})_{k\geq1}$ is bounded, for $n$ large enough, we have
\begin{equation*}
\begin{aligned}
&\frac{-\log\mu(I_{n})}{-\log|I_{n}|}\\
=&\frac{-\log\mu(I_{m_{k}})+(t(m_{k}-a_{i_{k}}-1)+l)\log b}{-\log|I_{m_{k}}|+(t(m_{k}-a_{i_{k}})+l)\log b}\\
=&\frac{-\log\mu(I_{m_{k}})+(t(m_{k}-a_{i_{k}}-1)+l)\log b}{-\log|I_{m_{k}}|+(t(m_{k}-a_{i_{k}}-1)+l)\log b}\cdot\frac{-\log|I_{m_{k}}|+(t(m_{k}-a_{i_{k}}-1)+l)\log b}{-\log|I_{m_{k}}|+(t(m_{k}-a_{i_{k}})+l)\log b}\\
\geq&\frac{-\log\mu(I_{m_{k}})+(t(m_{k}-a_{i_{k}}-1)+l)\log b}{-\log|I_{m_{k}}|+(t(m_{k}-a_{i_{k}}-1)+l)\log b}\cdot\frac{-\log|I_{m_{k}}|}{-\log|I_{m_{k}}|+t\log b}\\
>&(1-\varepsilon)\frac{\log\mu(I_{m_{k}})}{\log|I_{m_{k}}|},
\end{aligned}
\end{equation*}
where we have used the fact that
\begin{equation*}
\frac{a+x}{b+x}\geq\frac{a}{b}\ {\rm for\ all}\ 0<a\leq b,x\geq0.
\end{equation*}
Thus, for any $\varepsilon>0$, we have
\begin{equation*}
\liminf_{n\to\infty}\frac{\log\mu(I_{n})}{\log|I_{n}|}\geq(1-\varepsilon)\liminf_{k\to\infty}\frac{\log\mu(I_{m_{k}})}{\log|I_{m_{k}}|}.
\end{equation*}
Letting $\varepsilon\to0$, we obtain that
\begin{equation*}
\liminf_{n\to\infty}\frac{\log\mu(I_{n})}{\log|I_{n}|}\geq\liminf_{k\to\infty}\frac{\log\mu(I_{m_{k}})}{\log|I_{m_{k}}|}.
\end{equation*}
Thus,
\begin{equation}\label{221}
\liminf_{n\to\infty}\frac{\log\mu(I_{n})}{\log|I_{n}|}=\liminf_{k\to\infty}\frac{\log\mu(I_{m_{k}})}{\log|I_{m_{k}}|}=\frac{\theta-1-\theta\hat{v}}{(\theta-1)(\theta\hat{v}+1)}.
\end{equation}

Secondly, for each $x\in E_{\theta,\hat{v}}$, we prove that
\begin{equation}\label{222}
\liminf_{r\to0}\frac{\log\mu(B(x,r))}{\log r}=\liminf_{n\to\infty}\frac{\log\mu(I_{n}(x))}{\log|I_{n}(x)|},
\end{equation}
where $I_{n}(x)$ is the $n$th basic interval containing the point $x$. For any $r\in(0,b^{-a_{i_{2}}})$, there exists $n\geq a_{i_{2}}$, such that
\begin{equation*}
b^{-(n+1)}<r\leq b^{-n}.
\end{equation*}
We show that $I_{n+1}(x)\subset B(x,r)$ and $B(x,r)\cap E_{\theta,\hat{v}}$ is contained in three $n$th basic interval. For every $y\in I_{n+1}(x)$, we have
\begin{equation*}
|y-x|\leq b^{-(n+1)}<r.
\end{equation*}
Thus,
\begin{equation*}
I_{n+1}(x)\subset B(x,r).
\end{equation*}
Denote by $I_{n}'$ ($I_{n}''$) the nearest basic interval to the left (right) of $I_{n}(x)$, respectively. We prove that
\begin{equation*}
B(x,r)\cap E_{\theta,\hat{v}}\subset I_{n}'\cup I_{n}(x)\cup I_{n}''.
\end{equation*}
For all $y\in B(x.r)\cap E_{\theta,\hat{v}}$. Since $E_{\theta,\hat{v}}$ is contained in the union of all $n$th basic interval, we have $y$ belongs to the union of all $n$th basic interval. If $y\notin I_{n}'\cup I_{n}(x)\cup I_{n}''$, then
\begin{equation*}
|y-x|>b^{-n}\geq r.
\end{equation*}
This contradicts with $y\in B(x.r)\cap E_{\theta,\hat{v}}$. Therefore,
\begin{equation*}
B(x,r)\cap E_{\theta,\hat{v}}\subset I_{n}'\cup I_{n}(x)\cup I_{n}''.
\end{equation*}
Thus,
\begin{equation*}
\mu(I_{n+1}(x))\leq \mu(B(x,r))=\mu(B(x,r)\cap E_{\theta,\hat{v}})\leq 3\cdot\mu(I_{n}(x)).
\end{equation*}
By the above inequalities, we have established ($\ref{222}$). In view of ($\ref{221}$) and ($\ref{222}$), we have established ($\ref{220}$). Finally, by the mass distribution principle \cite[p.26]{KF}, we obtain that
\begin{equation*}
\dim_{\rm H}(E_{\theta,\hat{v}})\geq\frac{\theta-1-\theta\hat{v}}{(\theta-1)(\theta\hat{v}+1)}.
\end{equation*}
Thus,
\begin{equation*}
\dim_{\rm H}(\hat{\mathcal{V}}_{b,A}^{\ast}(\hat{v})\cap\mathcal{V}_{b,A}^{\ast}(\theta\hat{v}))\geq\frac{\theta-1-\theta\hat{v}}{(\theta-1)(\theta\hat{v}+1)}
\end{equation*}
for every $\theta\geq\frac{1}{1-\hat{v}}$. It then follows from Theorem $\ref{14}$ that
\begin{equation*}
\dim_{\rm H}(\hat{\mathcal{V}}_{b,A}^{\ast}(\hat{v})\cap\mathcal{V}_{b,A}^{\ast}(\theta\hat{v}))=\dim_{\rm H}(\hat{\mathcal{V}}_{b,A} (\hat{v})\cap\mathcal{V}_{b,A}^{\ast}(\theta\hat{v}))=\frac{\theta-1-\theta\hat{v}}{(\theta-1)(\theta\hat{v}+1)}
\end{equation*}
for any $\theta\geq\frac{1}{1-\hat{v}}$.
\textrm{}
\end{proof}
\section{Proof of Theorem $\ref{11}$}\label{40}
\begin{proof}[Proof of Theorem $\ref{11}$]
In view of Theorem $\ref{14}$, for any $\hat{v}\in(0,1)$, we have
\begin{equation*}
\hat{\mathcal{V}}_{b,A}^{\ast}(\hat{v})=\bigcup_{\theta\geq\frac{1}{1-\hat{v}}}\left(\hat{\mathcal{V}}_{b,A}^{\ast}(\hat{v})\cap\mathcal{V}_{b,A}^{\ast}(\theta\hat{v})\right).
\end{equation*}
It follows form Theorem $\ref{15}$ that
\begin{equation}\label{31}
\dim_{\rm H}(\hat{\mathcal{V}}_{b,A}^{\ast}(\hat{v}))\geq\sup_{\theta\geq\frac{1}{1-\hat{v}}}\frac{\theta-1-\theta\hat{v}}{(\theta-1)(\theta\hat{v}+1)}.
\end{equation}
A rapid calculation shows that the right hand side of ($\ref{31}$) is a continuous function of the parameter $\theta$ on the interval $[\frac{1}{1-\hat{v}},+\infty)$, reaching its maximum at the point $\theta_{0}=\frac{2}{1-\hat{v}}$. Therefore,
\begin{equation*}
\dim_{\rm H}(\hat{\mathcal{V}}_{b,A}^{\ast}(\hat{v}))\geq\sup_{\theta\geq\frac{1}{1-\hat{v}}}\frac{\theta-1-\theta\hat{v}}{(\theta-1)(\theta\hat{v}+1)}=\frac{\theta_{0}-1-\theta_{0}\hat{v}}{(\theta_{0}-1)(\theta_{0}\hat{v}+1)}=\left(\frac{1-\hat{v}}{1+\hat{v}}\right)^{2}.
\end{equation*}
On the other hand, by $\eqref{19}$, we know that
\begin{equation*}
\dim_{\rm H}(\hat{\mathcal{V}}_{b,A}^{\ast}(\hat{v}))\leq\dim_{\rm H}(\hat{\mathcal{V}}_{b,A}(\hat{v}))\leq\left(\frac{1-\hat{v}}{1+\hat{v}}\right)^{2}.
\end{equation*}
Thus,
\begin{equation*}
\dim_{\rm H}(\hat{\mathcal{V}}_{b,A}^{\ast}(\hat{v}))=\dim_{\rm H}(\hat{\mathcal{V}}_{b,A}(\hat{v}))=\left(\frac{1-\hat{v}}{1+\hat{v}}\right)^{2}
\end{equation*}
for any $\hat{v}\in[0,1]$.
\textrm{}
\end{proof}

\section{Proof of Theorems $\ref{12}$ and $\ref{13}$}\label{50}
The following lemma is the key to proving Theorem $\ref{12}$.
\begin{lemma}\label{41}
Let $\hat{v}\in(0,\eta)$. Assume that $\lim\limits_{n\to\infty}\frac{a_{n+1}}{a_{n}}$ exists and $1<\lim\limits_{n\to\infty}\frac{a_{n+1}}{a_{n}}<+\infty$. If $l$ is an non-negative integer such that $\frac{\eta^{l}-1}{\hat{v}}<\eta^{l}$, then for any $\theta\in \left(\frac{\eta^{l}-1}{\hat{v}},\eta^{l}\right)$, we have
\begin{equation*}
\hat{\mathcal{V}}_{b,A}(\hat{v})\cap\mathcal{V}_{b,A}^{\ast}(\theta\hat{v})=\emptyset.
\end{equation*}
\end{lemma}
\begin{proof}
Suppose that there exists $\theta\in\left(\frac{\eta^{l}-1}{\hat{v}},\eta^{l}\right)$, such that
\begin{equation*}
\hat{\mathcal{V}}_{b,A}(\hat{v})\cap\mathcal{V}_{b,A}^{\ast}(\theta\hat{v})\neq\emptyset.
\end{equation*}
Let $\xi\in\hat{\mathcal{V}}_{b,A}(\hat{v})\cap\mathcal{V}_{b,A}^{\ast}(\theta\hat{v})$, then
\begin{equation*}
\hat{v}_{b,A}(\xi)\geq\hat{v}\ {\rm and}\ v_{b,A}(\xi)=\theta\hat{v}.
\end{equation*}
By the proof of Lemma $\ref{21}$, we know that there exists a subsequence $(a_{i_{k}})_{k\geq1}$ of $(a_{k})_{k\geq1}$ and a subsequence $(m_{k})_{k\geq1}\subset\mathbb{N}$ of positive integers, with
\begin{equation*}
a_{i_{k}}+2\leq m_{k}\leq a_{i_{k+1}}+1,\ m_{k+1}-a_{i_{k+1}}>m_{k}-a_{i_{k}}\ {\rm for\ any}\ k\in\mathbb{N},
\end{equation*}
such that
\begin{equation*}
v_{b,A}(\xi)=\lim_{k\to\infty}\frac{m_{k}-a_{i_{k}}}{a_{i_{k}}}\ {\rm and}\ \hat{v}_{b,A}(\xi)=\liminf_{k\to\infty}\frac{m_{k}-a_{i_{k}}}{a_{i_{k+1}-1}}.
\end{equation*}
Thus,
\begin{equation}\label{42}
\lim_{k\to\infty}\frac{m_{k}-a_{i_{k}}}{a_{i_{k}}}=\theta\hat{v}
\end{equation}
and
\begin{equation}\label{43}
\liminf_{k\to\infty}\frac{m_{k}-a_{i_{k}}}{a_{i_{k+1}}}\geq\eta^{-1}\hat{v}.
\end{equation}
It follows from ($\ref{42}$) and ($\ref{43}$) that
\begin{equation*}
\limsup_{k\to\infty}\frac{a_{i_{k+1}}}{a_{i_{k}}}\leq\eta\theta<\eta^{l+1}.
\end{equation*}
Since $\lim\limits_{n\to\infty}\frac{a_{n+1}}{a_{n}}=\eta$, we have $\lim\limits_{k\to\infty}\frac{a_{i_{k}+l+1}}{a_{i_{k}}}=\eta^{l+1}$. Thus,
\begin{equation*}
i_{k+1}\leq i_{k}+l
\end{equation*}
for all $k$ large enough. By ($\ref{42}$), we have
\begin{equation*}
\lim_{k\to\infty}\frac{m_{k}}{a_{i_{k}}}=1+\theta\hat{v}>\eta^{l}.
\end{equation*}
Let $\varepsilon=1+\theta\hat{v}-\eta^{l}$, for all $k$ large enough, we obtain that
\begin{equation*}
m_{k}-a_{i_{k+1}}>\left(1+\theta\hat{v}-\frac{1}{4}\varepsilon\right)a_{i_{k}}-a_{i_{k}+l}>\left(1+\theta\hat{v}-\frac{1}{4}\varepsilon\right)a_{i_{k}}-\left(\eta^{l}+\frac{1}{4}\varepsilon\right)a_{i_{k}}= \frac{\varepsilon}{2}a_{i_{k}}.
\end{equation*}
It contradicts with $m_{k}-a_{i_{k+1}}\leq1$ for any $k\geq1$. Hence,
\begin{equation*}
\hat{\mathcal{V}}_{b,A}(\hat{v})\cap\mathcal{V}_{b,A}^{\ast}(\theta\hat{v})=\emptyset.
\end{equation*}
\textrm{}
\end{proof}

Now, we give the proof of Theorem $\ref{12}$.
\begin{proof}[Proof of Theorem $\ref{12}$]
Since
\begin{equation*}
\hat{v}\in\bigcup_{l=l_{0}}^{\infty}\left(\eta-\frac{2\eta}{\eta^{l}+1},\eta-\frac{2}{\eta^{l}}\right)\ {\rm and}\ l_{1}=\left\lfloor\frac{\log2-\log(\eta-\hat{v})}{\log \eta}\right\rfloor+1,
\end{equation*}
we have $l_{1}\geq l_{0}$,
\begin{equation}\label{44}
\eta-\frac{2\eta}{\eta^{l_{1}}+1}<\hat{v}<\eta-\frac{2}{\eta^{l_{1}}},
\end{equation}
and
\begin{equation}\label{45}
\frac{\eta^{l_{1}}-1}{\hat{v}}<\frac{2}{\eta-\hat{v}}<\eta^{l_{1}}.
\end{equation}
Let $l_{2}$ be the smallest positive integer such that $\eta-\frac{1}{\eta^{l_{2}}}\geq\hat{v}$. By the second inequality of ($\ref{44}$) and the definition of $l_{2}$, we have $l_{1}\geq l_{2}$. As the proofs for the case $l_{2}=1$ and the case $l_{2}\geq2$ are similar, we assume that $l_{2}\geq2$ and leave the details of the proof for the case $l_{2}=1$ to the reader.

According to the definition of $l_{2}$, for all $1\leq l\leq l_{2}-1$, we have $\eta^{l}<\frac{1}{\eta-\hat{v}}$. In view of Lemma $\ref{41}$, we have
\begin{equation}\label{46}
\dim_{\rm H}(\hat{\mathcal{V}}_{b,A}(\hat{v}))=\dim_{\rm H}\left(\bigcup_{l=l_{2}}^{\infty}\bigcup_{\theta\in\left[\eta^{l},\frac{\eta^{l+1}-1}{\hat{v}}\right]}\hat{\mathcal{V}}_{b,A}(\hat{v})\cap\mathcal{V}_{b,A}^{\ast}(\theta\hat{v})\right).
\end{equation}
We consider two cases.\\
Case 1: $\eta^{l_{2}}>\frac{1}{\eta-\hat{v}}$. For any $\rho>0$ with
\begin{equation*}
\rho\hat{v}^{-1}<\frac{1}{2}\min\left\{\eta^{l_{2}}-\frac{1}{\eta-\hat{v}},\eta^{l_{1}}-\frac{2}{\eta-\hat{v}}\right\},
\end{equation*}
we have
\begin{equation}\label{47}
\begin{aligned}
&\quad\bigcup_{l=l_{2}}^{\infty}\bigcup_{\theta\in\left[\eta^{l},\frac{\eta^{l+1}-1}{\hat{v}}\right]}\left(\hat{\mathcal{V}}_{b,A}(\hat{v})\cap\mathcal{V}_{b,A}^{\ast}(\theta\hat{v})\right)\\
&\subset\bigcup_{l=l_{2}}^{\infty}\bigcup_{\theta\in\left[\eta^{l}-\rho\hat{v}^{-1},\frac{\eta^{l+1}-1}{\hat{v}}\right]\cap\mathbb{Q}}\left(\{\xi\in\mathbb{R}:\hat{v}_{b,A}(\xi)\geq\hat{v}\}\cap\{\xi\in\mathbb{R}:\theta\hat{v}\leq v_{b,A}(\xi)\leq\theta\hat{v}+\rho\}\right).
\end{aligned}
\end{equation}
Furthermore, by Theorem $\ref{14}$, for $\rho>0$, we have
\begin{equation}\label{48}
\begin{aligned}
&\quad\dim_{\rm H}\left(\{\xi\in\mathbb{R}:\hat{v}_{b,A}(\xi)\geq\hat{v}\}\cap\{\xi\in\mathbb{R}:\theta\hat{v}\leq v_{b,A}(\xi)\leq\theta\hat{v}+\rho\}\right)\\
&\leq\frac{\hat{v}(\eta\theta-1-\theta\hat{v})+\rho(\eta-\hat{v})}{(1+\theta\hat{v})((\eta\theta-1)\hat{v}+\rho\eta)}.
\end{aligned}
\end{equation}
Combining ($\ref{46}$), ($\ref{47}$), ($\ref{48}$) and the countable stability of Hausdorff dimension, we obtain that
\begin{equation*}
\begin{aligned}
\dim_{\rm H}(\hat{\mathcal{V}}_{b,A}(\hat{v}))&\leq\sup_{l\geq l_{2}}\sup_{\theta\in \left[\eta^{l}-\rho\hat{v}^{-1},\frac{\eta^{l+1}-1}{\hat{v}}\right]\cap\mathbb{Q}}\frac{\hat{v}(\eta\theta-1-\theta\hat{v})+\rho(\eta-\hat{v})}{(1+\theta\hat{v})((\eta\theta-1)\hat{v}+\rho\eta)}\\
&\leq\sup_{l\geq l_{2}}\sup_{\theta\in \left[\eta^{l}-\rho\hat{v}^{-1},\frac{\eta^{l+1}-1}{\hat{v}}\right]}\frac{\hat{v}(\eta\theta-1-\theta\hat{v})+\rho(\eta-\hat{v})}{(1+\theta\hat{v})((\eta\theta-1)\hat{v}+\rho\eta)}\\
&\leq\sup_{l\geq l_{2}}\sup_{\theta\in\left[\eta^{l}-\rho\hat{v}^{-1},\frac{\eta^{l+1}-1}{\hat{v}}\right]}\frac{\eta\theta-1-\theta\hat{v}}{(\eta\theta-1)(1+\theta\hat{v}) }+\frac{\rho(\eta-\hat{v})^{3}}{\eta\hat{v}^{2}}.
\end{aligned}
\end{equation*}
Let $f(\theta)=\frac{\eta\theta-1-\theta\hat{v}}{(1+\theta\hat{v})(\eta\theta-1)}, \theta\in[\frac{1}{\eta-\hat{v}},+\infty)$. An easy calculation  shows that $f(\theta)$ is strictly increasing in $[\frac{1}{\eta-\hat{v}},\frac{2}{\eta-\hat{v}}]$  and strictly decreasing in $[\frac{2}{\eta-\hat{v}},+\infty)$. \\
If $l_{2}=l_{1}$, then the supremum is attained at $\theta = \eta^{l_{1}}-\rho\hat{v}^{-1}$, hence
\begin{equation*}
\begin{aligned}
&\quad\dim_{\rm H}(\hat{\mathcal{V}}_{b,A}(\hat{v}))\\
&\leq\frac{\eta(\eta^{l_{1}}-\rho\hat{v}^{-1})-1-(\eta^{l_{1}}-\rho\hat{v}^{-1})\hat{v}}{(\eta(\eta^{l_{1}}-\rho\hat{v}^{-1})-1)(1+(\eta^{l_{1}}-\rho\hat{v}^{-1})\hat{v}) }+\frac{\rho(\eta-\hat{v})^{3}}{\eta\hat{v}^{2}}.
\end{aligned}
\end{equation*}
If $l_{2}<l_{1}$, then the supremum is attained at either $\theta = \eta^{l_{1}}-\rho\hat{v}^{-1}$ or $\theta =\frac{\eta^{l_{1}}-1}{\hat{v}}$, hence
\begin{equation*}
\begin{aligned}
&\quad\dim_{\rm H}(\hat{\mathcal{V}}_{b,A}(\hat{v}))\\
&\leq\max\left\{\frac{\eta(\eta^{l_{1}}-\rho\hat{v}^{-1})-1-(\eta^{l_{1}}-\rho\hat{v}^{-1})\hat{v}}{(\eta(\eta^{l_{1}}-\rho\hat{v}^{-1})-1)(1+(\eta^{l_{1}}-\rho\hat{v}^{-1})\hat{v}) },\frac{\eta\frac{\eta^{l_{1}}-1}{\hat{v}}-1-\frac{\eta^{l_{1}}-1}{\hat{v}}\hat{v}}{(\eta\frac{\eta^{l_{1}}-1}{\hat{v}}-1)(1+\frac{\eta^{l_{1}}-1}{\hat{v}}\hat{v}) }\right\}+\frac{\rho(\eta-\hat{v})^{3}}{\eta\hat{v}^{2}}.
\end{aligned}
\end{equation*}
Letting $\rho\to0$, we obtain that
\begin{equation*}
\dim_{\rm H}(\hat{\mathcal{V}}_{b,A}(\hat{v}))\leq \begin{cases} \frac{\eta^{l_{1}+1}-1-\eta^{l_{1}}\hat{v}}{(\eta^{l_{1}+1}-1)(1+\eta^{l_{1}}\hat{v})}, &\textit{if}\ l_{2}=l_{1},\\
\max\left\{\frac{\eta^{l_{1}+1}-1-\eta^{l_{1}}\hat{v}}{(\eta^{l_{1}+1}-1)(1+\eta^{l_{1}}\hat{v}) },\frac{\eta^{l_{1}}-1-\eta^{l_{1}-1}\hat{v}}{\eta^{l_{1}-1}(\eta(\eta^{l_{1}}-1)-\hat{v})}\right\}, &\textit{if}\ l_{2}<l_{1}.\end{cases}
\end{equation*}
Hence,
\begin{equation*}
\dim_{\rm H}(\hat{\mathcal{V}}_{b,A}(\hat{v}))\leq\max\left\{\frac{\eta^{l_{1}+1}-1-\eta^{l_{1}}\hat{v}}{(\eta^{l_{1}+1}-1)(1+\eta^{l_{1}}\hat{v}) },\frac{\eta^{l_{1}}-1-\eta^{l_{1}-1}\hat{v}}{\eta^{l_{1}-1}(\eta(\eta^{l_{1}}-1)-\hat{v})}\right\}.
\end{equation*}
Case 2: $\eta^{l_{2}}=\frac{1}{\eta-\hat{v}}$. In this case, we must have $l_{2}<l_{1}$. It follows from Theorem $\ref{14}$ that
\begin{equation*}
\dim_{\rm H}\left(\hat{\mathcal{V}}_{b,A}(\hat{v})\cap\mathcal{V}_{b,A}^{\ast}\left(\frac{1}{\eta-\hat{v}}\hat{v}\right)\right)=0.
\end{equation*}
In view of $\eqref{46}$, we have
\begin{equation*}
\begin{aligned}
&\quad\dim_{\rm H}(\hat{\mathcal{V}}_{b,A}(\hat{v}))\\
&=\dim_{\rm H}\left(\bigcup_{\theta\in\left(\eta^{l_{2}},\frac{\eta^{l_{2}+1}-1}{\hat{v}}\right]}(\hat{\mathcal{V}}_{b,A}(\hat{v})\cap\mathcal{V}_{b,A}^{\ast}(\theta\hat{v}))\cup\bigcup_{l=l_{2}+1}^{\infty} \bigcup_{\theta\in\left[\eta^{l},\frac{\eta^{l+1}-1}{\hat{v}}\right]}(\hat{\mathcal{V}}_{b,A}(\hat{v})\cap\mathcal{V}_{b,A}^{\ast}(\theta\hat{v}))\right).
\end{aligned}
\end{equation*}
For any $\rho>0$ with
\begin{equation*}
\rho\hat{v}^{-1}<\frac{1}{2}\left(\eta^{l_{1}}-\frac{2}{\eta-\hat{v}}\right),
\end{equation*}
we have
\begin{equation*}
\begin{aligned}
&\quad\bigcup_{\theta\in\left(\eta^{l_{2}},\frac{\eta^{l_{2}+1}-1}{\hat{v}}\right]}(\hat{\mathcal{V}}_{b,A}(\hat{v})\cap\mathcal{V}_{b,A}^{\ast}(\theta\hat{v}))\cup\bigcup_{l=l_{2}+1}^{\infty} \bigcup_{\theta\in\left[\eta^{l},\frac{\eta^{l+1}-1}{\hat{v}}\right]}(\hat{\mathcal{V}}_{b,A}(\hat{v})\cap\mathcal{V}_{b,A}^{\ast}(\theta\hat{v}))\\
&\subset\bigcup_{\theta\in\left[\frac{1}{\eta-\hat{v}},\frac{\eta^{l_{2}+1}-1}{\hat{v}}\right]\cap\mathbb{Q}}\{\xi\in\mathbb{R}:\hat{v}_{b,A}(\xi)\geq\hat{v}\ {\rm and}\ \theta\hat{v}\leq v_{b,A}(\xi)\leq\theta\hat{v}+\rho\}\\
&\cup\bigcup_{l=l_{2}+1}^{\infty}\bigcup_{\theta\in\left[\eta^{l}-\rho\hat{v}^{-1},\frac{\eta^{l+1}-1}{\hat{v}}\right]\cap\mathbb{Q}}\{\xi\in\mathbb{R}:\hat{v}_{b,A}(\xi)\geq\hat{v}\ {\rm and}\ \theta\hat{v}\leq v_{b,A}(\xi)\leq\theta\hat{v}+\rho\}.
\end{aligned}
\end{equation*}
Similar with the proof of Case 1, we obtain that
\begin{equation*}
\dim_{\rm H}(\hat{\mathcal{V}}_{b,A}(\hat{v}))\leq\max\left\{\frac{\eta^{l_{1}+1}-1-\eta^{l_{1}}\hat{v}}{(\eta^{l_{1}+1}-1)(1+\eta^{l_{1}}\hat{v}) },\frac{\eta^{l_{1}}-1-\eta^{l_{1}-1}\hat{v}}{\eta^{l_{1}-1}(\eta(\eta^{l_{1}}-1)-\hat{v})}\right\}.
\end{equation*}
An easy calculation shows that
\begin{equation*}
\max\left\{\frac{\eta^{l_{1}+1}-1-\eta^{l_{1}}\hat{v}}{(\eta^{l_{1}+1}-1)(1+\eta^{l_{1}}\hat{v}) },\frac{\eta^{l_{1}}-1-\eta^{l_{1}-1}\hat{v}}{\eta^{l_{1}-1}(\eta(\eta^{l_{1}}-1)-\hat{v})}\right\}<\left(\frac{\eta-\hat{v}}{\eta+\hat{v}}\right)^{2}.
\end{equation*}
Thus,
\begin{equation*}
\begin{aligned}
\dim_{\rm H}(\hat{\mathcal{V}}_{b,A}(\hat{v}))&\leq\max\left\{\frac{\eta^{l_{1}+1}-1-\eta^{l_{1}}\hat{v}}{(\eta^{l_{1}+1}-1)(1+\eta^{l_{1}}\hat{v}) },\frac{\eta^{l_{1}}-1-\eta^{l_{1}-1}\hat{v}}{\eta^{l_{1}-1}(\eta(\eta^{l_{1}}-1)-\hat{v})}\right\}\\
&<\left(\frac{\eta-\hat{v}}{\eta+\hat{v}}\right)^{2}.
\end{aligned}
\end{equation*}
\textrm{}
\end{proof}

In order to prove Theorem $\ref{13}$, we need the following lemma.
\begin{lemma}\label{49}
Let $\hat{v}\in(0,\eta)$. Suppose that $\lim\limits_{n\to\infty}\frac{a_{n+1}}{a_{n}}$ exists and $1<\lim\limits_{n\to\infty}\frac{a_{n+1}}{a_{n}}<+\infty$, denote
\begin{equation*}
l'=l'(\hat{v}):=\max\left\{1,\left\lfloor\frac{-\log(\eta-\hat{v})}{\log\eta}\right\rfloor+1\right\}.
\end{equation*}
If $\theta\in\left[\eta^{l},\frac{\eta^{l+1}-1}{\hat{v}}\right)$ for some integer $l\geq l'$, we have
\begin{equation*}
\dim_{\rm H}(\hat{\mathcal{V}}_{b,A}(\hat{v})\cap\mathcal{V}_{b,A}^{\ast}(\theta\hat{v}))\geq\frac{\eta^{l+1}-1-\theta\hat{v}}{(\eta^{l+1}-1)(1+\theta\hat{v})}.
\end{equation*}
Furthermore, if $\theta = \eta^{l}$ for some integer $l\geq l'$, we have the following stronger result,
\begin{equation*}
\dim_{\rm H}\left(\hat{\mathcal{V}}_{b,A}(\hat{v})\cap\mathcal{V}_{b,A}^{\ast}(\eta^{l}\hat{v})\right)=\dim_{\rm H}\left(\hat{\mathcal{V}}_{b,A}^{\ast}(\hat{v})\cap\mathcal{V}_{b,A}^{\ast}(\eta^{l}\hat{v})\right)=\frac{\eta^{l+1}-1- \eta^{l}\hat{v}}{(\eta^{l+1}-1)(1+\eta^{l}\hat{v})}.
\end{equation*}
\end{lemma}
\begin{proof}
Our idea is to construct a suitable Cantor type subset of $\hat{\mathcal{V}}_{b,A}(\hat{v})\cap\mathcal{V}_{b,A}^{\ast}(\theta\hat{v})$. Since $1+\theta\hat{v}<\eta^{l+1}$, $\lim\limits_{n\to\infty}a_{n}=+\infty$ and $\lim\limits_{n\to\infty}\frac{a_{n+1}}{a_{n}}= \eta$, there exists $N\in\mathbb{N}$, such that for any $n\geq N$, we have
\begin{equation*}
a_{n}>3(\theta\hat{v})^{-1},\ a_{n+l+1}-a_{n}>(\theta\hat{v})^{-1}\ {\rm and}\ (1+\theta\hat{v})a_{n}\leq a_{n+l+1}-2.
\end{equation*}
Fix $i_{1}\geq N$, let $m_{1}=\lfloor(1+\theta\hat{v})a_{i_{1}}\rfloor$, $i_{2}=i_{1}+l+1$. Let $k\geq1$ be such that $i_{k}$ has been defined. Define $m_{k}=\lfloor(1+\theta\hat{v})a_{i_{k}}\rfloor$ and $i_{k+1}=i_{k}+l+1$. By the above construction, for each $k\geq1$, we have
\begin{equation*}
a_{i_{k}}+3\leq m_{k}\leq a_{i_{k+1}}-2,\  m_{k+1}-a_{i_{k+1}}>m_{k}-a_{i_{k}},
\end{equation*}
\begin{equation*}
\lim_{k\to\infty}\frac{m_{k}-a_{i_{k}}}{a_{i_{k}}}=\theta\hat{v}\ {\rm and}\ \lim_{k\to\infty}\frac{m_{k}-a_{i_{k}}}{a_{i_{k+1}}}=\theta\hat{v}\eta^{-l-1}\geq\eta^{-1}\hat{v}.
\end{equation*}

If $b\geq3$, let $E_{\theta,\hat{v}}$ be the set of all real numbers $\xi$ in $(0,1)$ whose $b$-ary expansion $\xi=\sum_{j\geq1}\frac{x_{j}}{b^{j}}$ satisfies, for each $k\geq1$,
\begin{equation*}
x_{a_{i_{k}}}=1,x_{a_{i_{k}}+1}=\ldots=x_{m_{k}-1}=0,x_{m_{k}}=1,
\end{equation*}
and
\begin{equation*}
x_{m_{k}+(m_{k}-a_{i_{k}})}=x_{m_{k}+2(m_{k}-a_{i_{k}})}=\ldots=x_{m_{k}+t_{k}(m_{k}-a_{i_{k}})}=1,
\end{equation*}
where $t_{k}$ is the largest integer such that $m_{k}+t_{k}(m_{k}-a_{i_{k}})<a_{i_{k+1}}$.

If $b=2$, let $E_{\theta,\hat{v}}$ be the set of all real numbers $\xi$ in $(0,1)$ whose $b$-ary expansion $\xi=\sum_{j\geq1}\frac{x_{j}}{b^{j}}$ satisfies, for each $k\geq1$,
\begin{equation*}
x_{a_{i_{k}}}=1,x_{a_{i_{k}}+1}=\ldots=x_{m_{k}-1}=0,x_{m_{k}}=1,
\end{equation*}
\begin{equation*}
x_{m_{k}+(m_{k}-a_{i_{k}})}=x_{m_{k}+2(m_{k}-a_{i_{k}})}=\ldots=x_{m_{k}+t_{k}(m_{k}-a_{i_{k}})}=1,
\end{equation*}
and
\begin{equation*}
x_{m_{k}+(m_{k}-a_{i_{k}})-1}=x_{m_{k}+2(m_{k}-a_{i_{k}})}=\ldots=x_{m_{k}+t_{k}(m_{k}-a_{i_{k}})-1}=0,
\end{equation*}
where $t_{k}$ is the largest integer such that $m_{k}+t_{k}(m_{k}-a_{i_{k}})<a_{i_{k+1}}$.

By the above construction, we have
\begin{equation*}
E_{\theta,\hat{v}}\subset\hat{\mathcal{V}}_{b,A}(\hat{v})\cap\mathcal{V}_{b,A}^{\ast}(\theta\hat{v}).
\end{equation*}
Furthermore, similar to the proof of Theorem $\ref{15}$, we have
\begin{equation*}
\dim_{\rm H}(E_{\theta,\hat{v}})\geq\frac{\eta^{l+1}-1-\theta\hat{v}}{(\eta^{l+1}-1)(1+\theta\hat{v})}.
\end{equation*}
Thus,
\begin{equation*}
\dim_{\rm H}(\hat{\mathcal{V}}_{b,A}(\hat{v})\cap\mathcal{V}_{b,A}^{\ast}(\theta\hat{v}))\geq\frac{\eta^{l+1}-1-\theta\hat{v}}{(\eta^{l+1}-1)(1+\theta\hat{v})}.
\end{equation*}
When $\theta=\eta^{l}$, we have
\begin{equation*}
\lim_{k\to\infty}\frac{m_{k}-a_{i_{k}}}{a_{i_{k}}}=\theta\hat{v}\ {\rm and}\ \lim_{k\to\infty}\frac{m_{k}-a_{i_{k}}}{a_{i_{k+1}}}=\theta\hat{v}\eta^{-l-1}=\eta^{-1}\hat{v}.
\end{equation*}
Thus, for any $\xi\in E_{\theta,\hat{v}}$, we have
\begin{equation*}
\hat{v}_{b,A}(\xi)=\hat{v}\ {\rm and}\ v_{b,A}(\xi)=\theta\hat{v}.
\end{equation*}
Therefore, when $\theta=\eta^{l}$, we have $E_{\theta,\hat{v}}\subset\hat{\mathcal{V}}_{b,A}^{\ast}(\hat{v})\cap\mathcal{V}_{b,A}^{\ast}(\theta\hat{v})$. Hence,
\begin{equation*}
\dim_{\rm H}\left(\hat{\mathcal{V}}_{b,A}^{\ast}(\hat{v})\cap\mathcal{V}_{b,A}^{\ast}(\eta^{l}\hat{v})\right)\geq\frac{\eta^{l+1}-1- \eta^{l}\hat{v}}{(\eta^{l+1}-1)(1+\eta^{l}\hat{v})}.
\end{equation*}
This together with Theorem $\ref{14}$ gives
\begin{equation*}
\dim_{\rm H}\left(\hat{\mathcal{V}}_{b,A}(\hat{v})\cap\mathcal{V}_{b,A}^{\ast}(\eta^{l}\hat{v})\right)=\dim_{\rm H}\left(\hat{\mathcal{V}}_{b,A}^{\ast}(\hat{v})\cap\mathcal{V}_{b,A}^{\ast}(\eta^{l}\hat{v})\right)=\frac{\eta^{l+1}-1- \eta^{l}\hat{v}}{(\eta^{l+1}-1)(1+\eta^{l}\hat{v})}.
\end{equation*}
\textrm{}
\end{proof}
\begin{proof}[Proof of Theorem $\ref{13}$]
By Lemma $\ref{49}$, we have
\begin{equation*}
\dim_{\rm H}(\hat{\mathcal{V}}_{b,A}^{\ast}(\hat{v}))\geq\frac{\eta^{l+1}-1-\eta^{l}\hat{v}}{(\eta^{l+1}-1)(1+\eta^{l}\hat{v})}
\end{equation*}
for all $l\in\mathbb{N}$ with $l\geq l'$. Hence,
\begin{equation}\label{410}
\dim_{\rm H}(\hat{\mathcal{V}}_{b,A}^{\ast}(\hat{v}))\geq\sup_{l\geq l'}\frac{\eta^{l+1}-1-\eta^{l}\hat{v}}{(\eta^{l+1}-1)(1+\eta^{l}\hat{v})}.
\end{equation}
An easy calculation shows that the right hand side of ($\ref{410}$) reaching its supremum at $l=\tilde{l}$, where
\begin{equation*}
\tilde{l}:=\max\left\{1,\left\lfloor\frac{\log(\eta+1)-\log(\eta-\hat{v})}{\log\eta}\right\rfloor\right\}.
\end{equation*}
Therefore,
\begin{equation*}
\dim_{\rm H}(\hat{\mathcal{V}}_{b,A}^{\ast}(\hat{v}))\geq\frac{\eta^{\tilde{l}+1}-1-\eta^{\tilde{l}}\hat{v}}{(\eta^{\tilde{l}+1}-1)(1+\eta^{\tilde{l}}\hat{v})}.
\end{equation*}
\textrm{}
\end{proof}

\section{Proof of Corollary $\ref{16}$}\label{60}
\begin{proof}[Proof of Corollary $\ref{16}$]
For any $\hat{v}=\eta-\frac{2}{\eta^{l}}$ with some integer $l\geq l_{0}$, we have
\begin{equation*}
\eta^{l}<\frac{\eta+1}{\eta-\hat{v}}=\frac{(\eta+1)\eta^{l}}{2}<\eta^{l+1}.
\end{equation*}
Thus, $\tilde{l}=l$. It follows from Theorem $\ref{13}$ that
\begin{equation*}
\dim_{\rm H}(\hat{\mathcal{V}}_{b,A}^{\ast}(\hat{v}))\geq\frac{\eta^{l+1}-1-\eta^{l}\hat{v}}{(\eta^{l+1}-1)(1+\eta^{l}\hat{v})}=\frac{\eta\cdot\frac{2}{\eta-\hat{v}}-1-\frac{2}{\eta-\hat{v}}\cdot\hat{v}}{(\eta\cdot\frac{2}{\eta-\hat{v}}-1)(1+\frac{2}{\eta-\hat{v}}\cdot\hat{v})}=\left(\frac{\eta-\hat{v}}{\eta+\hat{v}}\right)^{2}.
\end{equation*}
Since
\begin{equation*}
\dim_{\rm H}(\hat{\mathcal{V}}_{b,A}^{\ast}(\hat{v}))\leq\dim_{\rm H}(\hat{\mathcal{V}}_{b,A}(\hat{v}))\leq\left(\frac{\eta-\hat{v}}{\eta+\hat{v}}\right)^{2},
\end{equation*}
we have
\begin{equation*}
\dim_{\rm H}(\hat{\mathcal{V}}_{b,A}^{\ast}(\hat{v}))=\dim_{\rm H}(\hat{\mathcal{V}}_{b,A}(\hat{v}))=\left(\frac{\eta-\hat{v}}{\eta+\hat{v}}\right)^{2}=\frac{\eta^{\tilde{l}+1}-1-\eta^{\tilde{l}}\hat{v}}{(\eta^{\tilde{l}+1}-1)(1+\eta^{\tilde{l}}\hat{v})}.
\end{equation*}
For any
\begin{equation*}
\hat{v}\in\bigcup_{l=l_{0}}^{\infty}\left [\eta-\frac{\eta^{l}+\eta^{l +1}-1}{\eta^{2l}},\eta-\frac{2}{\eta^{l}}\right),
\end{equation*}
there exists $l\geq l_{0}$, such that
\begin{equation}\label{51}
\hat{v}\in\left[\eta-\frac{\eta^{l}+\eta^{l +1}-1}{\eta^{2l}},\eta-\frac{2}{\eta^{l}}\right).
\end{equation}
Since
\begin{equation*}
\eta-\frac{\eta^{l}+\eta^{l +1}-1}{\eta^{2l}}>\eta-\frac{\eta+1}{\eta^{l}}
\end{equation*}
and
\begin{equation*}
\eta-\frac{2}{\eta^{l}}<\eta-\frac{\eta+1}{\eta^{l+1}},
\end{equation*}
we have
\begin{equation*}
\eta-\frac{\eta+1}{\eta^{l}}<\hat{v}<\eta-\frac{\eta+1}{\eta^{l+1}}.
\end{equation*}
Thus,
\begin{equation}\label{52}
l=\left\lfloor\frac{\log(\eta+1)-\log(\eta-\hat{v})}{\log\eta}\right\rfloor=\tilde{l}.
\end{equation}
In view of Theorem $\ref{13}$, we have
\begin{equation}\label{53}
\dim_{\rm H}(\hat{\mathcal{V}}_{b,A}^{\ast}(\hat{v}))\geq\frac{\eta^{ l+1}-1-\eta^{ l}\hat{v}}{(\eta^{l+1}-1)(1+\eta^{l}\hat{v})}.
\end{equation}
On the other hand, an easy calculation shows that
\begin{equation*}
\eta-\frac{\eta^{l}+\eta^{l +1}-1}{\eta^{2l}}>\eta-\frac{2\eta}{\eta^{l}+1}.
\end{equation*}
By Theorem $\ref{12}$, we know that
\begin{equation*}
\dim_{\rm H}(\hat{\mathcal{V}}_{b,A}(\hat{v}))\leq\max\left\{\frac{\eta^{l+1}-1-\eta^{l}\hat{v}}{(\eta^{l+1}-1)(1+\eta^{l}\hat{v}) },\frac{\eta^{l}-1-\eta^{l-1}\hat{v}}{\eta^{l-1}(\eta(\eta^{l}-1)-\hat{v})}\right\}.
\end{equation*}
Note that
\begin{equation*}
\begin{aligned}
&\quad\frac{\eta^{l+1}-1-\eta^{l}\hat{v}}{(\eta^{l+1}-1)(1+\eta^{l}\hat{v}) }-\frac{\eta^{l}-1-\eta^{l-1}\hat{v}}{\eta^{l-1}(\eta(\eta^{l}-1)-\hat{v})}\\
&=\frac{\eta^{3l}\hat{v}^{2}-\eta^{l}(\eta^{l}-1)(\eta^{l+1}+\eta^{l}-1)\hat{v}+(\eta^{l}-1)^{2}(\eta^{l+1}-1)}{(\eta^{l+1}-1)(1+\eta^{l}\hat{v}) \eta^{l-1}(\eta(\eta^{l}-1)-\hat{v})}.
\end{aligned}
\end{equation*}
It follows from $\eqref{51}$ that $\eta(\eta^{l}-1)>\hat{v}$. What is more, the roots of polynomial
\begin{equation*}
\eta^{3l}t^{2}-\eta^{l}(\eta^{l}-1)(\eta^{l+1}+\eta^{l}-1)t+(\eta^{l}-1)^{2}(\eta^{l+1}-1)
\end{equation*}
are
\begin{equation*}
t_{1}=\frac{\eta^{l}-1}{\eta^{l}}\ {\rm and}\ t_{2}=\eta-\frac{\eta^{l}+\eta^{l +1}-1}{\eta^{2l}}.
\end{equation*}
Since $t_{1}<t_{2}\leq\hat{v}$,
\begin{equation*}
\frac{\eta^{l+1}-1-\eta^{l}\hat{v}}{(\eta^{l+1}-1)(1+\eta^{l}\hat{v})}\geq\frac{\eta^{l}-1-\eta^{l-1}\hat{v}}{\eta^{l-1}(\eta(\eta^{l}-1)-\hat{v})}.
\end{equation*}
Therefore,
\begin{equation}\label{54}
\dim_{\rm H}(\hat{\mathcal{V}}_{b,A}(\hat{v}))\leq\frac{\eta^{l+1}-1-\eta^{l}\hat{v}}{(\eta^{l+1}-1)(1+\eta^{l}\hat{v})}.
\end{equation}
Combining ($\ref{52}$), ($\ref{53}$) and ($\ref{54}$), we obtain that
\begin{equation*}
\dim_{\rm H}(\hat{\mathcal{V}}_{b,A}^{\ast}(\hat{v}))=\dim_{\rm H}(\hat{\mathcal{V}}_{b,A}(\hat{v}))=\frac{\eta^{\tilde{l}+1}-1-\eta^{\tilde{l}}\hat{v}}{(\eta^{\tilde{l}+1}-1)(1+\eta^{\tilde{l}}\hat{v})}.
\end{equation*}
\textrm{}
\end{proof}

\section{Some examples}\label{70}
\textbf{Example 7.1} Let $A=(2^{n})_{n=1}^{\infty}$, then $\eta=2$ and thus $l_{0}=2$. For any $\hat{v}\in(\frac{6}{5},\frac{3}{2})=\left(2-\frac{2\times2}{2^{2}+1},2-\frac{2}{2^{2}}\right)$, we have $l_{1}=2$. If follows from Theorem $\ref{12}$ that
\begin{equation*}
\begin{aligned}
\dim_{\rm H}(\hat{\mathcal{V}}_{b,A}(\hat{v}))&\leq\max\left\{\frac{2^{2+1}-1-2^{2}\hat{v}}{(2^{2+1}-1)(1+2^{2}\hat{v})},\frac{2^{2}-1-2^{2-1}\hat{v}}{2^{2-1}(2(2^{2}-1)- \hat{v})}\right\}\\
&=\max\left\{\frac{7-4\hat{v}}{7(1+4\hat{v})},\frac{3-2\hat{v}}{2(6-\hat{v})}\right\}.
\end{aligned}
\end{equation*}
An easy calculation shows that
\begin{equation*}
\tilde{l}=\begin{cases} 1, &{\rm if}\ \hat{v}\in(\frac{6}{5},\frac{5}{4}),\\ 2, &{\rm if}\ \hat{v}\in[\frac{5}{4},\frac{3}{2}). \end{cases}
\end{equation*}
By Theorem $\ref{13}$, we have
\begin{equation*}
\dim_{\rm H}(\hat{\mathcal{V}}_{b,A}^{\ast}(\hat{v}))\geq\begin{cases} \frac{2^{1+1}-1-2^{1}\hat{v}}{(2^{1+1}-1)(1+2^{1}\hat{v})}=\frac{3-2\hat{v}}{3(1+2\hat{v})}, &{\rm if}\ \hat{v}\in(\frac{6}{5},\frac{5}{4}),\\ \frac{2^{2+1}-1-2^{2}\hat{v}}{(2^{2+1}-1)(1+2^{2}\hat{v})}=\frac{7-4\hat{v}}{7(1+4\hat{v})}, &{\rm if}\  \hat{v}\in[\frac{5}{4},\frac{3}{2}).\end{cases}
\end{equation*}
In view of Corollary $\ref{16}$, we obtain that
\begin{equation*}
\dim_{\rm H}(\hat{\mathcal{V}}_{b,A}^{\ast}(\hat{v}))=\dim_{\rm H}(\hat{\mathcal{V}}_{b,A}(\hat{v}))=\frac{7-4\hat{v}}{7(1+4\hat{v})},\ \forall\ \hat{v}\in\left[\frac{21}{16},\frac{3}{2}\right].
\end{equation*}
\textbf{Example 7.2} Let $A=(3^{n})_{n=1}^{\infty}$, then $\eta=3$ and thus $l_{0}=1$. For any $\hat{v}\in(\frac{3}{2},\frac{7}{3})=(3-\frac{2\times3}{3^{1}+1},3-\frac{2}{3^{1}})$, we have $l_{1}=1$. If follows from Theorem $\ref{12}$ that
\begin{equation*}
\begin{aligned}
\dim_{\rm H}(\hat{\mathcal{V}}_{b,A}(\hat{v}))&\leq\max\left\{\frac{3^{2}-1-3^{1}\hat{v}}{(3^{2}-1)(1+3^{1}\hat{v})},\frac{3^{1}-1-3^{1-1}\hat{v}}{3^{1-1}(3(3^{1}-1)-\hat{v})}\right\}\\
&=\max\left\{\frac{8-3\hat{v}}{8(1+3\hat{v})},\frac{2-\hat{v}}{6-\hat{v}}\right\}.
\end{aligned}
\end{equation*}
An easy calculation shows that $\tilde{l}=1$. By Theorem $\ref{13}$, we have
\begin{equation*}
\dim_{\rm H}(\hat{\mathcal{V}}_{b,A}^{\ast}(\hat{v}))\geq\frac{3^{2}-1-3^{1}\hat{v}}{(3^{2}-1)(1+3^{1}\hat{v})}=\frac{8-3\hat{v}}{8(1+3\hat{v})},\ \forall\ \hat{v}\in\left(\frac{3}{2},\frac{7}{3}\right).
\end{equation*}
In view of Corollary $\ref{16}$, we obtain that
\begin{equation*}
\dim_{\rm H}(\hat{\mathcal{V}}_{b,A}^{\ast}(\hat{v}))=\dim_{\rm H}(\hat{\mathcal{V}}_{b,A}(\hat{v}))=\frac{8-3\hat{v}}{8(1+3\hat{v})},\ \forall\ \hat{v}\in\left[\frac{16}{9},\frac{7}{3}\right].
\end{equation*}
\subsection*{Acknowledgements}
We are grateful to the anonymous referee for a careful reading and many helpful comments. This research  was supported by NSFC 12271176, Guangdong Natural Science Foundation 2023A1515010691 and Guangdong Basic and Applied Basic Research Foundation 2023A1515110272.

\end{document}